\documentclass[3p]{elsarticle}
\usepackage{graphics}
\usepackage{graphicx}
\usepackage{epsfig}
\usepackage{amssymb}
\usepackage{amsthm}
\usepackage{lineno}
\usepackage{amsmath}
\numberwithin{equation}{section}
\usepackage{mathrsfs}
\usepackage{color}

\theoremstyle{plain}
\newtheorem{theorem}{Theorem}

\newtheorem{lemma}[theorem]{Lemma}
\newtheorem{proposition}[theorem]{Proposition}
\newtheorem{definition}[theorem]{Definition}

\numberwithin{equation}{section} \numberwithin{theorem}{section}
\newproof{pot}{{\bf{Proof of Theorem \ref{Thm1.1}}\rm}}
\newenvironment{remark}[1][Remark:]{\begin{trivlist}
		\item[\hskip \labelsep {\bfseries #1}]}{\end{trivlist}}

\begin{document}
	\begin{frontmatter}
		\author{Xiaojun Chang\corref{cor2}}
		\ead{changxj100@nenu.edu.cn}
		\cortext[cor2]{Corresponding author.}
		\address{School of Mathematics and Statistics \& Center for Mathematics and Interdisciplinary Sciences,\\
			Northeast Normal University, Changchun 130024, Jilin,
			China}

        \author{Hichem Hajaiej}
        \ead{hhajaie@calstatela.edu}
		\address{Department of Mathematics, Cal State LA, Los Angeles CA 90032, USA}

		\author{Zhouji Ma}
		\ead{mazj588@nenu.edu.cn}
		\address{School of Mathematics and Statistics \& Center for Mathematics and Interdisciplinary Sciences,\\
			Northeast Normal University, Changchun 130024, Jilin,
			China}

		\author{Linjie Song}
		\ead{songlinjie18@mails.ucas.edu.cn}
		\address{Institute of Mathematics, AMSS, Chinese Academy of Science, Beijing 100190, China\\
			University of Chinese Academy of Science, Beijing 100049, China}

\title{Existence and instability of standing waves for the biharmonic nonlinear Schr\"odinger equation with combined nonlinearities}

\begin{abstract}
 We consider the following biharmonic nonlinear Schr\"odinger equation with combined nonlinearities
\begin{equation*}
\begin{aligned}
i\partial_{t}\psi-\Delta^{2}\psi+\mu|\psi|^{q-2}\psi+|\psi|^{p-2}\psi=0, \qquad(t,x)\in\mathbb{R}\times\mathbb{R}^{N},
\end{aligned}
\end{equation*}
where $N\geq5,\ \psi\in\mathbb{R}\times\mathbb{R}^{N}\rightarrow\mathbb{C},\ \mu>0,\ 2+\frac{8}{N}\leq q<p\leq4^*$ and $4^*=\frac{2N}{N-4}$. We prove the existence of normalized ground state solutions for the biharmonic Schr\"odinger equation with combined nonlinearities and show that all ground states correspond to the local minima of the associated energy functional restricted to the Pohozaev set. Moreover, we prove that the standing waves are strongly unstable by blowup. In particular, our results cover the Sobolev case $p=4^*=\frac{2N}{N-4}$. Our method is novel and innovative as previous ideas cannot be used in many cases under this study.
\end{abstract}
\end{frontmatter}

\section{Introduction and main results}\label{intro}
In the last decades, there has been a growing interest in the study of the biharmonic Schr\"odinger equation
\begin{equation}\label{eq000}
i\partial_{t}\psi+\gamma\Delta^{2}\psi+\epsilon\Delta\psi+f(|\psi|)\psi=0,
\end{equation}
where $N\geq5,\gamma\not=0,\ \epsilon=\pm1$ or $\epsilon=0$, $\psi:\mathbb{R}\times\mathbb{R}^{N}\rightarrow\mathbb{C}$. In \cite{K1996, KS2000}, the biharmonic Schr\"odinger equation with mixed-dispersion term was introduced by Karpman and Shagalov to  capture the effects of small fourth-order dispersion terms in the propagation of intense laser beams in a bulk medium with Kerr nonlinearity. Ben-Artzi et al. \cite{Hbe} obtained necessary Strichartz estimates for the biharmonic operator $\Delta^2$. Here, we are particularly interested in the case $\epsilon=0$. This kind of biharmonic Schr\"odinger equations was first introduced in \cite{IK1983,T1985} to study the stability of solitons in magnetic materials once the effective quasi particle mass becomes infinite. Fibich, Ilan and Papanicolaou \cite{FIB} described some properties of the biharmonic Schr\"odinger equation in the subcritical case. Pausader \cite{PAU2} proved the global well-posedness and scattering theory under some conditions for the mass critical biharmonic Schr\"odinger equation. The mass supercritical and Sobolev subcritical regime were considered in \cite{GUO}. For the global well-posedness and scattering of the focusing Sobolev critical biharmonic Schr\"odinger equations, one can refer to \cite{MIA} and \cite{PAU}. For more related studies, the reader can see \cite{GCH, MZJ, Phan2018, ZJY, ZHU} for the biharmonic Schr\"odinger equation, and \cite{GOU, GOUL, DJSN, DBO, BFJ2019, FER, FIB} for the biharmonic Schr\"odinger equation with mixed-dispersion term.

In this paper, we are concerned with the existence and orbital instability of ground state solutions having prescribed mass, also called normalized ground state solutions for the following biharmonic nonlinear Schr\"odinger equation with combined nonlinearities
\begin{equation}\label{eq01}
i\partial_{t}\psi-\Delta^{2}\psi+\mu|\psi|^{q-2}\psi+|\psi|^{p-2}\psi=0, \quad\psi(0,x)=\psi_{0}\in H^2(\mathbb{R}^N),\quad(t,x)\in\mathbb{R}\times\mathbb{R}^{N},
\end{equation}
where $N\geq5,\ \psi:\mathbb{R}\times\mathbb{R}^{N}\rightarrow\mathbb{C},\ \mu>0,\ 2+\frac{8}{N} \leq q < p \leq 4^*$ and $4^* = \frac{2N}{N-4}$.

 Firstly, we recall that the standing waves of \eqref{eq01}, are solutions of the form $\psi(t,x)=e^{i\omega t}u(x),\ \omega\in\mathbb{R}$. Then the ansatz yields the elliptic equation
 \begin{equation}\label{eq02}
    \Delta^{2}u+\omega u-\mu|u|^{q-2}u-|u|^{p-2}u=0,\quad\mathrm{in}\ \mathbb{R}^{N}.
\end{equation}
While looking for solutions to \eqref{eq02}, one reasonable approach is to fix $\omega\in\mathbb{R}$ and search for the critical points of the following action functional
$$A_{p,q}(u)=\frac{1}{2}\|\Delta u\|_{2}^{2}+\frac{\omega}{2}\|u\|_{2}^{2}-\frac{\mu}{q}\|u\|_{q}^{q}-\frac{1}{p}\|u\|_{p}^{p}.$$
 Alternatively, we can search for the solutions to \eqref{eq02} with prescribed mass, which is a meaningful problem from both mathematical and physical perspectives. Indeed, the mass of a quantum system is a fundamental property that is related to the probability density of the system. In quantum mechanics, the squared modulus of the wave function represents the probability density of finding the system in a particular state. For normalized solutions, the total probability of finding the system in any state is conserved over time, as the mass is conserved.
By considering solutions to the biharmonic Schr\"odinger equation with prescribed mass, we can gain insights into the properties of the system in a consistent and meaningful way. Furthermore, studying the solutions to the biharmonic Schr\"odinger equation with prescribed mass
  can lead to interesting mathematical challenges and the discovery of new tools and insights.

We define on $H^{2}=H^{2}(\mathbb{R}^{N},\mathbb{C})$  the energy functional
$$E_{p,q}(u):=\frac{1}{2}\|\Delta u\|_{2}^{2}-\frac{\mu}{q}\|u\|_{q}^{q}-\frac{1}{p}\|u\|_{p}^{p}.$$
It is standard to check that $E_{p,q}$ is of class $C^{1}$ and that a critical point of $E_{p,q}$ restricted to the mass constraint
$$S(c)=\{u\in H^{2}:\|u\|_{2}^{2}=c\}$$
corresponds to a solution of \eqref{eq02} with prescribed mass $\|u\|_{2}^{2}=c$. Here, the parameter $\omega\in\mathbb{R}$ is unknown and it appears as a Lagrange multiplier. In this paper, we are interested in the solutions that minimize $E_{p,q}(u)$ among all non-trivial solutions, namely ground state solutions. We define it in the following way:
\begin{definition}\label{blp}
We say that a solution $u_{c}\in S(c)$ to \eqref{eq02} is a ground state solution to \eqref{eq02} if it is a solution with the least energy among all the solutions that belong to $S(c)$. Namely, it satisfies
$$E_{p,q}(u_{c})=\inf\{E_{p,q}(u), u\in S(c),(E_{p,q}|_{S(c)})^{'}(u)=0\}.$$
\end{definition}

In recent years, the study of the normalized ground state solutions of nonlinear Schr\"odinger type equations has been a hot topic. Particularly, there have been several results for the mixed-dispersion fourth-order Schr\"odinger equations. However, there are few references to the corresponding study of biharmonic Schr\"odinger equation with mixed nonlinearities in the literature. Recently, in \cite{MZJ}, the authors proved the existence of the normalized ground state solutions to \eqref{eq01} for $2<q<2+\frac{8}{N}<p=4^*$. In this paper, we will study the normalized ground state solutions to \eqref{eq01} in the remaining regimes.

Under the assumption $2+\frac{8}{N}\leq q<p\leq4^*$, $E_{p,q}$ restricted to $S(c)$ is unbounded from below. To overcome this difficulty, we construct a subset of $S(c)$ on which $E_{p,q}$ is bounded from below and coercive. We exploit a natural constraint called Pohozaev set of $E_{p,q}$ restricted to $S(c)$, in which all the critical points of $E_{p,q}|_{S(c)}$ belong to this Pohozaev set. The set is defined as follows
$$\mathcal{Q}_{p,q}(c)=\{u\in S(c):Q_{p,q}(u)=0\},$$
 where
$$Q_{p,q}(u)=\|\Delta u\|_{2}^{2}-\mu\gamma_{q}\|u\|_{q}^{q}-\gamma_{p}\|u\|_{p}^{p}.$$
with
$$\gamma_{r}=\frac{N(r-2)}{4r}=\frac{N}{2}(\frac{1}{2}-\frac{1}{r}), \forall r \in (2,4^*].$$
Hence it is natural to consider the minimization problem on this Pohozaev set
$$m_{p,q}(c)=\inf\limits_{u\in\mathcal{Q}_{p,q}(c)}E_{p,q}(u).$$

Then we state the following existence results.
\begin{theorem}\label{Thm1.1}
Let $N\geq5$, $c>0$, $\mu>0$, $2+\frac{8}{N}\leq q<p<4^*$. If $q=2+\frac{8}{N}$, we further assume that $\mu c^{\frac{4}{N}}<\frac{N+4}{NC^q_{N,q}}$. Then $E_{p,q}|_{S(c)}$ has a critical point $u$ at a positive level $E_{p,q}(u)>0$ with the following properties: $u$ solves \eqref{eq02} for some $\omega>0$, and is a normalized ground state of \eqref{eq02} on $S(c)$.
\end{theorem}

\begin{theorem}\label{Thm1.2}
	Let $N\geq5$, $c>0$, $\mu>0$ and $2+\frac{8}{N}\leq q < 4^*$. If $q=2+\frac{8}{N}$, we further assume that $\mu c^{\frac{4}{N}}<\frac{N+4}{NC^q_{N,q}}$. If $N = 5$, we further assume that $q > \frac{22}{3}$. Then $E_{4^*,q}|_{S(c)}$ has a critical point $u$ at a level $E_{4^*,q}(u) \in (0,\frac{2}{N}\mathcal{S}^{\frac{N}{4}})$ with the following properties: $u$ solves \eqref{eq02} for some $\omega>0$, and is a normalized ground state of \eqref{eq02} on $S(c)$ for $p = 4^*$.
\end{theorem}

 To search for the normalized ground state solutions of \eqref{eq01}, the first challenging thing is to construct a bounded Palais-Smale sequence $\{u_{n}\}$ of $E_{p,q}|_{S(c)}$. To overcome this difficulty, we prove the existence of a Palais-Smale sequence $\{{u_n}\}$ of $E_{p,q}|_{S(c)}$ at level $m_{p,q}(c)$ by constructing it within the subspace $\mathcal{Q}_{p,q}(c)$, where the boundedness of the sequence can be obtained, see Lemma \ref{cpq}. In fact, motivated by the work of Bartsch and Soave \cite{bar,bar3}, we first show that $t_{u}$ defined in Lemma \ref{uni} is unique, and the map $u\mapsto t_{u}$ is of $C^1$ as stated in Lemma \ref{rem}. This implies that the functional $\Psi(u)$ defined in Section \ref{ground states1} is also $C^1$, as mentioned in Lemma \ref{cc1}. By adapting the arguments used in \cite{bar} to the $C^1$ constrained functional $I:=\Psi(u)|_{S(c)}$, we can obtain the desired Palais-Smale sequence.

In addition, another major obstruction is faced because of the lack of the compactness for the obtained sequence $\{u_{n}\}$. For the Sobolev subcritical case, the key ingredient is to show that the map $c\rightarrow m_{p,q}(c)$ is strictly decreasing, as stated in Lemma \ref{dea}. In fact, up to a translation and up to a subsequence, we can see that $\{u_{n}\}$ has a nontrivial weak limit $u$ with $\|u\|_{2}^{2}=c_{1}\in(0,c]$, which satisfies
$$\Delta^{2}u+\omega u-\mu_{n}|u|^{q-2}u-|u|^{p-2}u=0,$$
and $u\in \mathcal{Q}_{p,q}(c_{1})$. By using some arguments, we can show that $\lim\limits_{n\rightarrow\infty}E_{p,q}(u_{n}-u)\geq0$ and it follows that $$m_{p,q}(c)=\lim\limits_{n\rightarrow\infty}E_{p,q}(u_{n})=E_{p,q}(u)+\lim\limits_{n\rightarrow\infty}E_{p,q}(u_{n}-u)\geq m_{p,q}(c_1).$$
This allows us to get the strong convergence of the sequence $\{u_{n}\}$ if $c\rightarrow m_{p,q}(c)$ is strictly decreasing for $c>0$. However, in the Sobolev critical case, the problem becomes more complicated. We need to construct an appropriate testing function to derive an upper bound for $m_{4^,q}(c)$, which is denoted as $m_{4^,q}(c)<\frac{2}{N}\mathcal{S}^{\frac{N}{4}}$ and is obtained in Lemma \ref{ub}. This upper bound provides an estimate on the energy level $m_{4^*,q}(c)$, which is crucial for further analysis and convergence results.

Together with above analysis, we obtain a non-existence result as follows.
\begin{theorem}\label{Thm1.3}
	Let $N\geq5$, $c>0$, $\mu>0$, $2+\frac{8}{N} = q < p \leq 4^*$. We further assume that $\mu c^{\frac{4}{N}} \geq \frac{N+4}{NC^q_{N,q}}$. Then $E_{p,q}|_{S(c)}$ has no critical point $u$ at the level $E_{p,q}(u) = m_{p,q}(c)$ (including the case $\mathcal{Q}_{p,q}^\mu(c)=\emptyset$).
\end{theorem}

Finally in this section, we will discuss the strong instability of standing waves to \eqref{eq01}. First we recall the definition of instability below.
\begin{definition}\label{inst}
We say that a solution $u\in H^{2}({\mathbb{R}^{N}})$ to \eqref{eq02} is unstable by blowup in finite time, if for any $\epsilon>0$, there exists a $v\in H^{2}({\mathbb{R}^{N}})$ such that $\|v-u\|_{H^2}<\epsilon$ and the solution $\psi(t)$ to \eqref{eq01} with initial data $\psi_{0}=v$ blows up in finite time.
\end{definition}

The instability of standing waves to nonlinear Schr\"odinger equations has been extensively studied in the literature, we refer the reader to \cite{H-S, Coz, SOA, SOA2} and the references therein. The strategy for proving instability is often based on the variational characterization of the ground states on the Pohozaev or Nehari manifold, see \cite{BOU} for example. The key point is to show the following inequality
$$Q_{p,q}(\psi)\leq E_{p,q}(\psi_{0})-E_{p,q}(u),$$
where $u$ is the corresponding ground state. Then by using the virial identity, we can obtain a differential inequality as follows
$$\frac{d^2}{dt^2}\int_{\mathbb{R}^{N}}|x\psi(t)|^2dx=8Q_{p,q}(\psi(t))\leq-(E_{p,q}(u)-E_{p,q}(\psi_{0}))<0.$$
This implies that the solution $\psi(t)$ blows up in finite time.

While for the biharmonic nonlinear Schr\"odinger equation with mixed dispersion
\begin{equation}\label{eq230219}
i\partial\psi-\gamma\Delta^2\psi+\epsilon\Delta\psi+|\psi|^{2\sigma}\psi=0,\qquad \gamma>0,\ \epsilon\in\mathbb{R},\ 4\leq\sigma\leq\frac{4}{N-4},	
\end{equation}
the virial identity is not available due to the appearance of the operator $\Delta^2$.
Boulenger and Lenzmann \cite{BOU} used a localized version of the virial identity from \cite{MIA,PAU} to prove blowup results for radial solutions in $H^{2}(\mathbb{R}^N)$ of \eqref{eq230219} with $\gamma=1$ under various criteria in the mass-critical, mass-supercritical and Sobolev critical cases. In particular when $\epsilon=0$, they assumed that the energy of the initial data was either negative or non-negative with some additional conditions on the initial data. Bonheure et al. \cite{GOUL} proved the instability by blowup of the radial ground state solutions for \eqref{eq230219} when $\epsilon=0$ in both the mass-critical and mass-supercritical regimes. Later, in \cite{GOU}, they showed that the instability of the radial normalized ground state solutions can be proved by blowup with $\epsilon=1$ when $4\leq\sigma<\frac{4}{N-4}$.

In the present paper, we aim to study the instability by blowup for ground states of \eqref{eq01}. To the best of our knowledge, so far, there has been no existing instability results for the biharmonic nonlinear Schr\"odinger equations with combined nonlinearities. We will borrow an idea from Boulenger and Lenzmann \cite{BOU} to study this issue. However, we cannot directly use the arguments from Boulenger and Lenzmann's work
  due to the fact that the energy of the initial data satisfies $E_{p,q}(\psi_{0})>0$ without any additional assumptions on $\psi_{0}$. Therefore, we need to come up with a different approach to establish instability results for \eqref{eq01}.

More precisely, in order to study the instability of \eqref{eq01}, one of the main challenges is to obtain the following estimate
\begin{equation}\label{eq230218}
\frac{d}{dt}M_{\varphi_{R}}[\psi(t)]\leq-\delta\|\Delta\psi(t)\|^2_2, \qquad	 \delta>0,
\end{equation}
where the modified energy $M_{\varphi_{R}}[\psi(t)]$ is defined in \eqref{eq6001}. This estimate is crucial for establishing the instability result. However, there are several terms in the inequality that make it difficult to estimate \eqref{eq230218}, including the term $\mu|u|_q^q$ or $|u|_p^p$ that remains in the inequality after applying variational arguments.
Another challenge is that the term $|\Delta\psi(t)|^2_2$ in \eqref{eq230218} may be unbounded, which further complicates the estimation process. To overcome these challenges, a key point is to carefully choose a suitable bound for $|\Delta\psi(t)|^2_2$, which can be used to
obtain the desired estimate for $\frac{d}{dt}M_{\varphi_{R}}[u(t)]$.

In what follows, we state our main instability results in Sobolev subcritical case and Sobolev critical case respectively.
\begin{theorem}\label{orbi}
Let $2+\frac{8}{N}\leq q<p<4^*$. Then standing waves associated with the ground state solutions to \eqref{eq02} with prescribed mass are strongly unstable.
\end{theorem}
\begin{theorem}\label{orbi2}
Let $2+\frac{8}{N}\leq q<p=4^*$. Then standing waves associated with the ground state solutions to \eqref{eq01} with prescribed mass are strongly unstable.
\end{theorem}

\begin{remark}
When considering the instability of normalized ground states solutions to \eqref{eq01}, we work in the space $H^2_{r}(\mathbb{R}^{N})$. It is still not clear for us how to control the error terms without the radial symmetry, even though the local virial identities hold true without radial assumptions. If we set $p,q\in\mathbb{N}$, then we may use the symmetric-decreasing rearrangement in Fourier space to prove a radial symmetry result for ground states. See \cite{BOU,LEN1,LEN2}.
\end{remark}

This paper is organized as follows: In Section \ref{preliminary}, some preliminary results are presented. In Section \ref{ground states1}, we give the proof of Theorem \ref{Thm1.1}. In Section \ref{ground states2}, the proof of Theorem \ref{Thm1.2} is established. In Section \ref{non-existence}, we prove Theorem \ref{Thm1.3}. Finally in Section \ref{insta}, we give the proofs of Theorem \ref{orbi} and Theorem\ref{orbi2} which are used to study the strong instability of standing waves corresponding to the normalized ground states. $\|\cdot\|_q$ will denote the standard norm in $L^{q}(\mathbb{R}^N)$.

\section{Preliminaries}\label{preliminary}
In this section, we present some useful lemmas to be often used in this paper.

Take
$\|u\|_{H^{2}}=(\int_{\mathbb{R}^{N}}|\Delta u|^{2}+|u|^{2}dx)^{\frac{1}{2}}$
 to be the norm of $H^{2}$, which is equivalent to the standard one
$\|u\|_{H^{2}}=(\int_{\mathbb{R}^{N}}|D^{2}u|^{2}+|\nabla u|^{2}+|u|^{2}dx)^{\frac{1}{2}}.$
We first recall the well-known Gagliardo-Nirenberg inequality for any $u\in H^{2}$.
\begin{lemma}\label{gni}
Let $N\geq5$ and $2<r<4^*$. Then the following sharp Gagliardo-Nirenberg inequality
 \begin{equation}\label{eq05}
\|u\|_{r}^{r}\leq C^{r}_{N,r}\|\Delta u\|_{2}^{r\gamma_{r}}\|u\|_{2}^{r(1-\gamma_{r})},
\end{equation}	
holds for any $u\in H^{2}(\mathbb{R}^N)$, where $C_{N,r}$ denotes the sharp constant.	
\end{lemma}

Then, by similar arguments as in \cite{M1996}, we have the following Lions' type lemma in $H^{2}(\mathbb{R}^{N})$.
\begin{lemma}\label{pll}
Let $R>0$. If $\{u_{n}\}$ is bounded in $H^{2}$ and if
$$\sup\limits_{y\in\mathbb{R}^{N}}\int_{B(y,R)}|u_{n}|^{2}dx\rightarrow0,\quad\mathrm{as}\quad n\rightarrow\infty,$$
then $u_{n}\rightarrow0$ in $L^{r}(\mathbb{R}^{N})$ for $2<p<4^*$.
\end{lemma}




For any $u\in S(c)$, we define
\begin{equation}\label{eq06}
u_{t}(x)=t^{\frac{N}{4}}	u(\sqrt{t}x).
\end{equation}
Clearly, $u_{t}\in S(c)$ for any $t>0$. Furthermore,
\begin{equation}\label{eq07}
E_{p,q}(u_{t})=\frac{1}{2}t^{2}\|\Delta u\|_{2}^{2}-\frac{\mu}{q}t^{\frac{N(q-2)}{4}}\|u\|_{q}^{q}-\frac{1}{p}t^{\frac{N(p-2)}{4}}\|u\|_{p}^{p}
\end{equation}
and
$$Q_{p,q}(u_{t})=t^{2}\|\Delta u\|_{2}^{2}-\mu\gamma_{q}t^{\frac{N(q-2)}{4}}\|u\|_{q}^{q}-\gamma_{p}t^{\frac{N(p-2)}{4}}\|u\|_{p}^{p}.$$
Then one has the following properties for $E_{p,q}(u_{t})$ and $Q_{p,q}(u_{t})$.
\begin{lemma}\label{uni}
	Let $N\geq5$, $c>0$, $\mu>0$ and $2+\frac{8}{N}\leq q<p\leq4^*$. If $q=2+\frac{8}{N}$, we further assume that $\mu c^{\frac{4}{N}}<\frac{N+4}{NC^q_{N,q}}$. Then for any $u\in S(c)$, there exists a unique $t_{u}\in(0,+\infty)$ for which  $u_{t_{u}}\in\mathcal{Q}_{p,q}(c)$. Moreover, $t_{u}$ is the unique critical point of $E_{p,q}(u_{t})$ such that $E_{p,q}(u_{t_{u}})=\max_{t\in(0,+\infty)}E_{p,q}(u_{t})$ and the function $u\mapsto E_{p,q}(u_{t_{u}})$ is concave  on $[t_{u},+\infty)$. In particular, if $Q_{p,q}(u)\leq0$, then $t_{u}\in(0,1]$.	
\end{lemma}
\begin{proof}
First we let $$Q_{p,q}(u_{t})=t^{2}(\|\Delta u\|_{2}^{2}-\mu\gamma_{q}t^{\frac{N(q-2)}{4}-2}\|u\|_{q}^{q}-\gamma_{p}t^{\frac{N(p-2)}{4}-2}\|u\|_{p}^{p})=t^{2}h(t).$$
 When $2+\frac{8}{N}<q<p\leq4^*$, we have $\frac{N(p-2)}{4}-2>\frac{N(q-2)}{4}-2>0$. When $2+\frac{8}{N}=q<p\leq4^*$ and $\mu c^{\frac{4}{N}}<\frac{N+4}{NC^q_{N,q}}$, we have $\frac{N(p-2)}{4}-2>\frac{N(q-2)}{4}-2=0$. Using the Gagliardo-Nirenberg inequality, we get
 $$\mu\gamma_{q}\|u\|_{q}^{q}\leq\mu\gamma_{q}C^{q}_{N,q}\|\Delta u\|_{2}^{q\gamma_{q}}\|u\|_{2}^{q(1-\gamma_{q})}<\|\Delta u\|_{2}^{2}.$$
Hence, for $2+\frac{8}{N}\leq q<p\leq4^*$, it always holds that $h(t)>0$ for $t$ small enough, $h(t)<0$ for $t$ large enough and $h'(t)<0$ for all $t\in(0,+\infty)$. So $Q_{p,q}(u_{t})$ has a unique zero $t_{u}$.

 By direct calculations, we have $\frac{d}{dt}
 E_{p,q}(u_{t})=t^{-1}Q_{p,q}(u_{t})$, $E_{p,q}(u_{t})>0$ for $t>0$ small enough and $\lim\limits_{t\rightarrow\infty}E_{p,q}(u_{t})=-\infty$. Thus, $t_{u}$ is the unique critical point of $E_{p,q}(u_{t})$ and $E_{p,q}(u_{t_{u}})=\max_{t\in(0,+\infty)}E_{p,q}(u_{t})$.
 Now writing $t=st_{u}$, we have
 \begin{align*}
 	\frac{d^2}{d^2t}E_{p,q}(u_{t})&=\|\Delta u\|_{2}^{2}-\mu\gamma_{q}(\frac{N(q-2)}{4}-1)t^{\frac{N(q-2)}{4}-2}\|u\|_{q}^{q}-\gamma_{p}(\frac{N(p-2)}{4}-1)t^{\frac{N(p-2)}{4}-2}\|u\|_{p}^{p}\\
 	&=\|\Delta u\|_{2}^{2}-\mu\gamma_{q}(\frac{N(q-2)}{4}-1)s^{\frac{N(q-2)}{4}-2}t_{u}^{\frac{N(q-2)}{4}-2}\|u\|_{q}^{q}\\
 	&-\gamma_{p}(\frac{N(p-2)}{4}-1)s^{\frac{N(p-2)}{4}-2}t_{u}^{\frac{N(p-2)}{4}-2}\|u\|_{p}^{p}\\
 	&=\frac{1}{t_{u}^{2}}\big[t_{u}^{2}\|\Delta u\|_{2}^{2}-\mu\gamma_{q}(\frac{N(q-2)}{4}-1)s^{\frac{N(q-2)}{4}-2}t_{u}^{\frac{N(q-2)}{4}}\|u\|_{q}^{q}\\
 	&-\gamma_{p}(\frac{N(p-2)}{4}-1)s^{\frac{N(p-2)}{4}-2}t_{u}^{\frac{N(p-2)}{4}}\|u\|_{p}^{p}\big].\\
 \end{align*}
 Since
 $$0=Q_{p,q}(u_{t_{u}})=t_{u}^{2}\|\Delta u\|_{2}^{2}-\mu\gamma_{q}t_{u}^{\frac{N(q-2)}{4}}\|u\|_{q}^{q}-\gamma_{p}t_{u}^{\frac{N(p-2)}{4}}\|u\|_{p}^{p},$$
we infer that $\frac{d^2}{d^2t}E_{p,q}(u_{t})<0$ for any $s\geq1$. In addition, if $Q_{p,q}(u)\leq0$, it follows that $0<t_{u}\leq1$. This completes the proof.
\end{proof}
\begin{lemma}\label{rem}
Let $2+\frac{8}{N}\leq q<p\leq4^*$. For any $u\in S(c)$, the map $u\mapsto t_{u}$ is of class $C^1$, where $t_{u}$ is defined in Lemma \ref{uni}.	
\end{lemma}
\begin{proof}
For $2+\frac{8}{N}\leq q<p\leq4^*$ and $u\in S(c)$, we consider the $C^1$ function $\Psi:\mathbb{R}\times S(c)\mapsto\mathbb{R}$ defined by $\Psi(t,u)=\psi'_{u}(t)$, where $\psi_{u}(t)=E_{p,q}(u_{t})$. By direct computation and the definition of $t_{u}$, we have
$$\Psi(t_{u},u)=\frac{1}{t_{u}}Q_{p,q}(u_{t_{u}})=0\quad\mathrm{and}\quad\partial_{t}\Psi(t_{u},u)=\psi''_{u}(t_{u})<0.$$
 By the Implicit Function Theorem, we can get that the map $u\mapsto t_{u}$ is of class $C^1$.
 \end{proof}

\begin{lemma}\label{cpq}
Let $N\geq5$, $c>0$, $\mu>0$ and $2+\frac{8}{N}\leq q<p\leq4^*$. If $q=2+\frac{8}{N}$, we further assume that $\mu c^{\frac{4}{N}}<\frac{N+4}{NC^q_{N,q}}$. Then,
\begin{description}
	\item $(i)\ \mathcal{Q}_{p,q}(c)\not=\emptyset;$
	\item $(ii)\ \inf\limits_{u\in\mathcal{Q}_{p,q}(c)}\|\Delta u\|_{2}^{2}>0;$
	\item $(iii)\ m_{p,q}(c)>0;$
	\item $(iv)\ E_{p,q}$ is coercive on $\mathcal{Q}_{p,q}(c)$, that is $E_{p,q}(u_{n})\rightarrow+\infty$ for any $\{u_{n}\}\subset\mathcal{Q}_{p,q}(c)$ with $\|u_{n}\|_{H^2}\rightarrow\infty$\textcolor[rgb]{0.7,0.04,0.8}{.}
\end{description}	
\end{lemma}
\begin{proof}
$(i)$\ By Lemma \ref{uni}, for any $u\in S(c)$, there always exists $t_{u}>0$ such that $u_{t_{u}}\in \mathcal{Q}_{p,q}(c)$, it follows that $\mathcal{Q}_{p,q}(c)\neq\emptyset$.

$(ii)$\ We divide the proof into two cases.

Case\ $1:$ $(p<4^*)\ $For any $u\in\mathcal{Q}_{p,q}(c)$, by the Gagliardo-Nirenberg inequality, we have
\begin{align*}\label{eq114}
\|\Delta u\|_{2}^{2}&=\mu\gamma_{q}\|u\|_{q}^{q}+\gamma_{p}\|u\|_{p}^{p}\\
&\leq\mu\gamma_{q}C^{q}_{N,q}\|u\|_{2}^{q(1-\gamma_{q})}\|\Delta u\|_{2}^{q\gamma_{q}}+\gamma_{p}C^{p}_{N,p}\|u\|_{2}^{p(1-\gamma_{p})}\|\Delta u\|_{2}^{p\gamma_{p}}\\
&=\mu\gamma_{q}C^{q}_{N,q}(\sqrt{c})^{q(1-\gamma_{q})}\|\Delta u\|_{2}^{q\gamma_{q}}+\gamma_{p}C^{p}_{N,p}(\sqrt{c})^{p(1-\gamma_{p})}\|\Delta u\|_{2}^{p\gamma_{p}}.\\
\end{align*}
If $2+\frac{8}{N}<q<p$, then $p\gamma_{p}>q\gamma_{q}>2$.  It follows that there exists a constant $C>0$ such that $\|\Delta u\|_{2}^{2}\geq C$.
If $2+\frac{8}{N}=q<p$ and $\mu c^{\frac{4}{N}}<\frac{N+4}{NC^q_{N,q}}$, then $p\gamma_{p}>q\gamma_{q}=2$ and $\frac{\mu NC^q_{N,q}}{N+4}c^{\frac{4}{N}}<1$. This also implies that there exists a constant $C>0$ such that $\|\Delta u\|_{2}^{2}\geq C$.

Case\ $2:$ $(p=4^*)\ $The proof here is similar to Case 1. The only difference is that we estimate the term $\int_{\mathbb{R}^{N}}|u|^{4^*}dx$ by using
$$\mathcal{S}\|u\|^{2}_{4^*}\leq\|\Delta u\|^{2}_{2}.$$

$(iii)$\ If $2+\frac{8}{N}<q<p$, we can easily get $\mu\gamma_{q}\|u\|_{q}^{q}+\gamma_{p}\|u\|_{p}^{p}=\|\Delta u\|_{2}^{2}\geq C$. If $2+\frac{8}{N}=q<p$ and $\mu c^{\frac{4}{N}}<\frac{N+4}{NC^q_{N,q}}$, it follows that
\begin{equation}\label{eq8801}
\gamma_{p}\|u\|_{p}^{p}\geq(1-\frac{\mu NC^q_{N,q}}{N+4}c^{\frac{4}{N}})\|\Delta u\|_{2}^{2}\geq C(1-\frac{\mu NC^q_{N,q}}{N+4}c^{\frac{4}{N}}).	
\end{equation}
Then, for any $u\in \mathcal{Q}_{p,q}(c)$, we have
\begin{align}\label{4-23-1}
E_{p,q}(u)&=(\frac{1}{2}-\frac{1}{q\gamma_{q}})\|\Delta u\|_{2}^{2}+(\frac{\gamma_{p}}{q\gamma_{q}}-\frac{1}{p})\|u\|_{p}^{p}\nonumber\\
&=(\frac{1}{2}-\frac{1}{p\gamma_{p}})\|\Delta u\|_{2}^{2}+(\frac{\gamma_{q}}{p\gamma_{p}}-\frac{1}{q})\mu\|u\|_{q}^{q}\nonumber\\
&=(\frac{\gamma_{p}}{2}-\frac{1}{p})\|u\|_{p}^{p}+(\frac{\gamma_{q}}{2}-\frac{1}{q})\mu\|u\|_{q}^{q}\ge C_1
\end{align}
for some $C_1>0$, which implies that $m_{p,q}(c)>0$.

$(iv)$\ By (\ref{4-23-1}) it is easily seen that the conclusion holds.
\end{proof}

Since for given $c>0$, one can see from Lemma \ref{cpq} that the infimum
$$m_{p,q}(c)=\inf\limits_{u\in\mathcal{Q}_{p,q}(c)}E_{p,q}(u)$$
is well-defined and strictly positive. We will characterize further the behavior of $m_{p,q}(c)$ when $c>0$ varies.

\begin{lemma}\label{dea}
The function $c\mapsto m_{p,q}(c)$ is non-increasing on $(0,+\infty)$.
\end{lemma}
\begin{proof}
First we set $\gamma(c)=\inf\limits_{u\in S(c)}\max\limits_{t>0}E_{p,q}(u_{t})$. We claim that $\gamma(c)=m_{p,q}(c)$.
Indeed, for any $u\in S(c)$, by Lemma \ref{uni}, there exists a $t_{u}>0$ such that $u_{t_{u}}\in\mathcal{Q}_{p,q}(c)$ and $\max\limits_{t>0}E_{p,q}(u_{t})=E_{p,q}(u_{t_{u}})\geq m_{p,q}(c)$, thus we get $\gamma(c)\geq m_{p,q}(c)$. On the other hand, for any $u\in\mathcal{Q}_{p,q}(c),\ \max\limits_{t>0}E_{p,q}(u_{t})=E_{p,q}(u)$, this immediately implies that $\gamma(c)\leq m_{p,q}(c)$. Then we conclude that $\gamma(c)=m_{p,q}(c)$.

Next we assume that $0<c_{1}<c_{2}<+\infty$. For any $u\in S(c_{1})$, define $v(x)=(\frac{c_{1}}{c_{2}})^{\frac{N-4}{8}}u((\frac{c_{1}}{c_{2}})^{\frac{1}{4}}x),\ \forall x\in \mathbb{R}^{N}$. then we have the followings:
$$\|v\|_{2}^{2}=c_{2},\quad \|\Delta v\|_{2}^{2}=\|\Delta u\|_{2}^{2},$$
$$\|v\|_{q}^{q}=(\frac{c_{1}}{c_{2}})^{\frac{q(N-4)-2N}{8}}\|u\|_{q}^{q},\quad\|v\|_{p}^{p}=(\frac{c_{1}}{c_{2}})^{\frac{p(N-4)-2N}{8}}\|u\|_{p}^{p}.$$
By the definition of $m_{p,q}(c)$ and the fact $\gamma(c)=m_{p,q}(c)$, we have
\begin{align*}
m_{p,q}(c_{2})&\leq\max_{t\in(0,+\infty)}E_{p,q}(v_{t})\\
&=E_{p,q}(v_{t_{v}})=E_{p,q}(t_{v}^{\frac{N}{4}}v(\sqrt{t_{v}}x))\\
&=E_{p,q}(t_{v}^{\frac{N}{4}}(\frac{c_{1}}{c_{2}})^{\frac{N-4}{8}}u((\frac{c_{1}}{c_{2}})^{\frac{1}{4}}\sqrt{t_{v}}x))\\
&=\frac{1}{2}\int_{\mathbb{R}^{N}}|\Delta u_{t_{v}}|^{2}dx-(\frac{c_{1}}{c_{2}})^{\frac{q(N-4)-2N}{8}}\frac{\mu}{q}\int_{\mathbb{R}^{N}}|u_{t_{v}}|^{q}dx-(\frac{c_{1}}{c_{2}})^{\frac{p(N-4)-2N}{8}}\frac{1}{p}\int_{\mathbb{R}^{N}}|u_{t_{v}}|^{p}dx\\
&<\frac{1}{2}\int_{\mathbb{R}^{N}}|\Delta u_{t_{v}}|^{2}dx-\frac{\mu}{q}\int_{\mathbb{R}^{N}}|u_{t_{v}}|^{q}dx-\frac{1}{p}\int_{\mathbb{R}^{N}}|u_{t_{v}}|^{p}dx\\
&=E_{p,q}(u_{t_{v}})\leq\max_{t\in(0,+\infty)}E_{p,q}(u_{t}),\\
\end{align*}
where $2+\frac{8}{N}\leq q<p\leq4^*$ and $N\geq5$. Since $u\in S(c_{1})$ is arbitrary, we have
\begin{equation}
	m_{p,q}(c_{2})\leq\inf_{u\in S(c_{1})}\max_{t\in(0,+\infty)}E_{p,q}(u_{t})=m_{p,q}(c_{1}).
\end{equation}
\end{proof}	
Next we will prove that $m_{p,q}(c)$ is continuous. Before that, we will show some lemmas which will be used in the proof for the continuity of $m_{p,q}(c)$.
\begin{lemma}\label{lio}
Let $2+\frac{8}{N}\leq q<p\leq4^*$. Let $\{u_{n}\}_{n\in\mathbb{N}^+}$ be a bounded sequence in $H^{2}(\mathbb{R}^{N})$ and $a_{n}$ be a sequence satisfying $\lim\limits_{n\rightarrow\infty}a_{n}=1$. Then
$$\lim\limits_{n\rightarrow\infty}\big|E_{p,q}(a_{n}u_{n})-E_{p,q}(u_{n})\big|=0.$$ 	
\end{lemma}
\begin{proof}
It follows from direct calculations that
$$\big|E_{p,q}(a_{n}u_{n})-E_{p,q}(u_{n})\big|\leq\big|\frac{(a_{n}^{2}-1)}{2}\|\Delta u_{n}\|_{2}^{2}\big|+\big|\frac{\mu}{q}(a_{n}^{q}-1)\|u_{n}\|_{q}^{q}\big|+\big|\frac{1}{p}(a_{n}^{p}-1)\|u_{n}\|_{p}^{p}\big|$$	
Since$\{u_{n}\}_{n\in\mathbb{R}}$ is bounded in $H^{2}(\mathbb{R}^{N})$, we can directly get the conclusion.
\end{proof}
\begin{lemma}\label{tpc}
Let $2+\frac{8}{N}\leq q<p\leq4^*$. If $q=2+\frac{8}{N}$, we further assume that $\mu c^{\frac{4}{N}}<\frac{N+4}{NC^q_{N,q}}$. For any $u\in S(c)$, let $t_{u}$ be be given as in Lemma \ref{uni}. Then $t_{u}$ satisfies
$$t_{u}^{\frac{N(p-2)}{4}-2}<C\frac{\|\Delta u\|_{2}^{2}}{\|u\|_{p}^{p}}.$$
 where $C$ is a constant independent of $u$.
\end{lemma}
\begin{proof}
For convenience, in this part, we set $f(t)=\mu|t|^{q-2}t+|t|^{p-2}t$, then $F(t)=\mu\frac{1}{q}|t|^{q}+\frac{1}{p}|t|^{p}$. We claim that
$$f(t)t>(2+\frac{8}{N})F(t)>0,\quad\mathrm{for}\ \mathrm{all}\ t\not=0.$$
Indeed, when $2+\frac{8}{N}<q<p\leq4^*$, we obviously have
$$f(t)t-(2+\frac{8}{N})F(t)=\mu(1-\frac{2+8/N}{q})|t|^{q}+(1-\frac{2+8/N}{p})|t|^{p}>0,\quad\mathrm{for}\ \mathrm{all}\ t\not=0.$$
When $2+\frac{8}{N}=q<p\leq4^*$, we still find that
$$f(t)t-(2+\frac{8}{N})F(t)=(1-\frac{2+8/N}{p})|t|^{p}>0,\quad\mathrm{for}\ \mathrm{all}\ t\not=0.$$
Therefore, we have the following inequality,
\begin{equation}\label{eq7701}
f(t)t-2F(t)>\frac{8}{N}F(t),\quad\mathrm{for}\ \mathrm{all}\ t\not=0.
\end{equation}
According to Lemma \ref{uni}, we have $u_{t_{u}}\in\mathcal{Q}_{p,q}(c)$. Hence $Q_{p,q}(u_{t_{u}})=0,\ i.e.$,
$$\|\Delta u_{t_{u}}\|_{2}^{2}=\mu\gamma_{q}\|u_{t_{u}}\|_{q}^{q}+\gamma_{p}\|u_{t_{u}}\|_{p}^{p}.$$
Then by the inequality in \eqref{eq7701}, we have
\begin{align*}
t_{u}^{2}\|\Delta u\|_{2}^{2}&=\mu\gamma_{q}\|u_{t_{u}}\|_{q}^{q}+\gamma_{p}\|u_{t_{u}}\|_{p}^{p}\\
&=\frac{N}{4}(\mu\frac{q-2}{q}\|u_{t_{u}}\|_{q}^{q}+\frac{p-2}{p}\|u_{t_{u}}\|_{p}^{p})\\
&=\frac{N}{4}\int_{\mathbb{R}^{N}}(f(u_{t_{u}})u_{t_{u}}-2F(u_{t_{u}}))dx\\
&>\frac{N}{4}\cdot\frac{8}{N}\int_{\mathbb{R}^{N}}F(u_{t_{u}})dx=2\int_{\mathbb{R}^{N}}F(u_{t_{u}})dx\\
&=2(\mu\frac{1}{q}\|u_{t_{u}}\|_{q}^{q}+\frac{1}{p}\|u_{t_{u}}\|_{p}^{p})\\
&=2(t_{u}^{\frac{N(q-2)}{4}}\mu\frac{1}{q}\|u\|_{q}^{q}+t_{u}^{\frac{N(p-2)}{4}}\frac{1}{p}\|u\|_{p}^{p})\\
&>2t_{u}^{\frac{N(p-2)}{4}}\frac{1}{p}\|u\|_{p}^{p}.\\
\end{align*}
Hence, we deduce that
$$t_{u}^{\frac{N(p-2)}{4}-2}<\frac{p}{2}\frac{\|\Delta u\|_{2}^{2}}{\|u\|_{p}^{p}}.$$
\end{proof}

\begin{lemma}\label{ckj}
Let $N\geq5$, $\mu>0$ and $2+\frac{8}{N}\leq q<p\leq4^*$. If $q=2+\frac{8}{N}$, we further assume that $\mu c^{\frac{4}{N}}<\frac{N+4}{NC^q_{N,q}}$. It holds that
\begin{description}
\item[$(i)$]
$c\mapsto m_{p,q}(c)$ is a continuous mapping.

\item[$(ii)$]Let $c\in(0,+\infty)$. We have $m_{p,q}(c)\leq m_{p,q}(\alpha)+m_{p,q}(c-\alpha)$ for all $\alpha\in(0,c)$. The inequality is strict if $m_{p,q}(\alpha)$ or $m_{p,q}(c-\alpha)$ is reached.
\end{description}
\end{lemma}
\begin{proof}
$(i)\ $	 Based on the above lemma, for any fixed $c>0$, $m_{p,q}(c-h)$ and $m_{p,q}(c+h)$ are monotonic and bounded as $h\rightarrow0^+$, so they have limits. Moreover, $m_{p,q}(c-h)\geq m_{p,q}(c)\geq m_{p,q}(c+h)$, thus we obtain
$$\lim\limits_{h\rightarrow0^+}m_{p,q}(c-h)\geq m_{p,q}(c)\geq\lim\limits_{h\rightarrow0^+}m_{p,q}(c+h).$$
It remains to show the other direction. We finish it by proving two claims.

\noindent Claim $1$: $\lim\limits_{h\rightarrow0^+}m_{p,q}(c-h)\leq m_{p,q}(c)$.

\noindent Take $u\in\mathcal{Q}_{p,q}(c)$, for any $h>0$, we set $u^{h}(x)=\sqrt{1-\frac{h}{c}}u(x)$. Noticing that $\|u^{h}\|_{2}^{2}=c-h$, by Lemma \ref{uni}, there exists $t_{h}>0$ such that $u^{h}_{t_{h}}\in\mathcal{Q}_{p,q}(c-h)$. By Lemma \ref{tpc}, we have
$$t_{h}^{\frac{N(p-2)}{4}-2}<\frac{p}{2}\frac{\|\Delta u^{h}\|_{2}^{2}}{\|u^{h}\|_{p}^{p}}=\frac{p}{2}\frac{(1-\frac{h}{c})}{(1-\frac{h}{c})^{\frac{p}{2}}}\frac{\|\Delta u\|_{2}^{2}}{\|u\|_{p}^{p}},$$
which implies that $t_{h}$ is bounded as $h\rightarrow0^+$. Hence $u_{t_{h}}$ is bounded in $H^{2}(\mathbb{R}^N)$. So according to Lemma \ref{lio}, we have
$$m_{p,q}(c-h)\leq E_{p,q}(u^{h}_{t_{h}})=E_{p,q}(\sqrt{1-\frac{h}{c}}u_{t_{h}})\leq E_{p,q}(u_{t_{h}})+o(h)\leq E_{p,q}(u)+o(h).$$
Then we prove the claim.

\noindent Claim $2$: $\lim\limits_{h\rightarrow0^+}m_{p,q}(c+h)\geq m_{p,q}(c)$.

\noindent Choose $u\in\mathcal{Q}_{p,q}(c+h)$ for $h>0$ and let $v^{h}(x)=\frac{1}{\sqrt{1+\frac{h}{c}}}u(x)$. Since $\|v^{h}\|_{2}^{2}=c$, similar to Claim $1$, there exists $t_{h}>0$ such that $v^{h}_{t_{h}}\in\mathcal{Q}_{p,q}(c)$ and
$$t_{h}^{\frac{N(p-2)}{4}-2}<\frac{p}{2}\frac{\|\Delta v^{h}\|_{2}^{2}}{\|v^{h}\|_{p}^{p}}=\frac{p}{2}\frac{(1+\frac{h}{c})^{\frac{p}{2}}}{(1+\frac{h}{c})}\frac{\|\Delta u\|_{2}^{2}}{\|u\|_{p}^{p}}.$$
So $u_{t_{h}}$ is bounded in $H^{2}(\mathbb{R}^N)$. Using Lemma \ref{lio},
it follows that
$$m_{p,q}(c)\leq E_{p,q}(v^{h}_{t_{h}})=E_{p,q}(\frac{1}{\sqrt{1+\frac{h}{c}}}u_{t_{h}})\leq E_{p,q}(u_{t_{h}})+o(h)\leq E_{p,q}(u)+o(h).$$
Hence, we prove that $c\mapsto m_{p,q}(c)$ is continuous.

$(ii)\ $ We observe that for fixed $\alpha\in(0,c)$, it is sufficient to prove that the following holds
\begin{equation}\label{eq17}
	\forall\theta\in(1,\frac{c}{\alpha}]:m_{p,q}(\theta\alpha)\leq\theta m_{p,q}(\alpha)
\end{equation}
and that, if $m_{p,q}(\alpha)$ is reached, the inequality is strict. In fact, if \eqref{eq17} holds then we have
\begin{align*}
m_{p,q}(c)&=\frac{c-\alpha}{c}m_{p,q}(c)+\frac{\alpha}{c}_{p,q}m_{p,q}(c)\\
&=\frac{c-\alpha}{c}m_{p,q}(\frac{c}{c-\alpha}(c-\alpha))+\frac{\alpha}{c}m_{p,q}(\frac{c}{\alpha}\alpha)\\	
&\leq m_{p,q}(c-\alpha)+m_{p,q}(\alpha),\\
\end{align*}
with a strict inequality if $m_{p,q}(\alpha)$ is reached. By the definition of $m_{p,q}(c)$, we notice that, for any $\epsilon>0$ sufficiently small, there exists a $u\in \mathcal{Q}_{p,q}(\alpha)$ such that
\begin{equation}\label{eq18}
E_{p,q}(u)\leq m_{p,q}(\alpha)+\epsilon.
\end{equation}
Since $Q_{p,q}(u)=0$, we define $\widetilde{u}(x)=(\frac{1}{\theta})^{\frac{1}{p-2}}u((\frac{1}{\theta})^{\frac{p}{N(p-2)}}x)$. Then
$$\int_{\mathbb{R}^{N}}|\widetilde{u}|^{2}dx=\theta\alpha,\quad\int_{\mathbb{R}^{N}}|\widetilde{u}|^{p}dx=\int_{\mathbb{R}^{N}}|u|^{p}dx,$$
$$\int_{\mathbb{R}^{N}}|\widetilde{u}|^{q}dx=\int_{\mathbb{R}^{N}}|\widetilde{u}|^{q}dx=(\frac{1}{\theta})^{\frac{q-p}{p-2}}\int_{\mathbb{R}^{N}}|u|^{q}dx\geq\int_{\mathbb{R}^{N}}|u|^{q}dx,$$
$$\int_{\mathbb{R}^{N}}|\Delta\widetilde{u}|^{2}dx=(\frac{1}{\theta})^{\frac{2N+p(4-N)}{N(p-2)}}\int_{\mathbb{R}^{N}}|\Delta u|^{2}dx\leq\int_{\mathbb{R}^{N}}|\Delta u|^{2}dx\textcolor{blue}{.}$$
Hence $Q_{p,q}(\widetilde{u})\leq0$. So there exists $t_{0}\in(0,1]$ such that $\widetilde{u}_{t_{0}}\in\mathcal{Q}_{p,q}(\theta\alpha)$ and for any $\theta>1$,
\begin{align*}
	m_{p,q}(\theta\alpha)&\leq E_{p,q}(\widetilde{u}_{t_{0}})=(\frac{1}{2}-\frac{1}{q\gamma_{q}})\int_{\mathbb{R}^{N}}|\Delta\widetilde{u}_{t_{0}}|^{2}dx+(\frac{\gamma_{p}}{q\gamma_{q}}-\frac{1}{p})\int_{\mathbb{R}^{N}}|\widetilde{u}_{t_{0}}|^{p}dx\\
	&=(\frac{1}{2}-\frac{1}{q\gamma_{q}})t_{0}^{2}\int_{\mathbb{R}^{N}}|\Delta\widetilde{u}|^{2}dx+(\frac{\gamma_{p}}{q\gamma_{q}}-\frac{1}{p})t_{0}^{\frac{Np}{4}-\frac{N}{2}}\int_{\mathbb{R}^{N}}|\widetilde{u}|^{p}dx\\
	&=(\frac{1}{2}-\frac{1}{q\gamma_{q}})t_{0}^{2}(\frac{1}{\theta})^{\frac{2N+p(4-N)}{N(p-2)}}\int_{\mathbb{R}^{N}}|\Delta u|^{2}dx\\
	&+(\frac{\gamma_{p}}{q\gamma_{q}}-\frac{1}{p})t_{0}^{\frac{Np}{4}-\frac{N}{2}}\int_{\mathbb{R}^{N}}|u|^{p}dx\\
	&\leq\theta(\frac{1}{2}-\frac{1}{q\gamma_{q}})\int_{\mathbb{R}^{N}}|\Delta u|^{2}dx+\theta(\frac{\gamma_{p}}{q\gamma_{q}}-\frac{1}{p})\int_{\mathbb{R}^{N}}|u|^{p}dx\\
	&\leq \theta E_{p,q}(u)<\theta(m_{p,q}(\alpha)+\epsilon).
\end{align*}
Since $\epsilon$ is arbitrary, we have that $m_{p,q}(\theta\alpha)\leq\theta m_{p,q}(\alpha)$. If $m_{p,q}(\alpha)$ is reached, we can choose $\epsilon=0$. Then the strict inequality follows.
\end{proof}
\begin{remark}
For Lemma \ref{ckj} $(ii)$, it is not used in the rest of the paper but is useful for future application. Although the key point for us to show the strong convergence is Lemma \ref{dea}, it is possible to prove the strong convergence by this property.
\end{remark}
\begin{lemma}\label{dea}
Let $2+\frac{8}{N}\leq q<p\leq4^*$, $\mu>0$ and $c>0$. If $q=2+\frac{8}{N}$, we further assume that $\mu c^{\frac{4}{N}}<\frac{N+4}{NC^q_{N,q}}$. If there exists $u\in S(c)$ such that $E_{p,q}(u)=m_{p,q}(c)$ and
\begin{equation}\label{eq9910}
	\Delta^{2}u+\omega u=\mu|u|^{q-2}u+u|^{p-2}u,
\end{equation}
then $Q_{p,q}(u)=0$ and $\omega>0$. Moreover, the function $c\mapsto m_{p,q}(c)$ is strictly decreasing in a right neighborhood of $c$.	
\end{lemma}
\begin{proof}
	The fact that any solution to \eqref{eq9910} satisfies $Q_{p,q}(u)=0$ is a direct consequence of the Pohozaev identity. Now we deduce from \eqref{eq9910}
\begin{equation}\label{eq9911}
\|\Delta u\|_{2}^{2}+\omega\|u\|	_{2}^{2}-\mu\|u\|_{q}^{q}-\|u\|_{p}^{p}=0.
\end{equation}
Combining \eqref{eq9911} with $Q_{p,q}(u)=0$, we have
$$\omega\|u\|_{2}^{2}=\mu(1-\gamma_{q})\|u\|_{q}^{q}+(1-\gamma_{p})\|u\|_{p}^{p}.$$
Since $0<\gamma_{q}<\gamma_{p}\leq1$, it follows that $\omega>0$. Then we let $u_{\lambda,t}(x)=t^{\frac{N}{4}}\sqrt{\lambda}u(\sqrt{t}x)$ for $\lambda,\ t>0$. We define
$$\mathcal{K}(\lambda,t)=E_{p,q}(u_{\lambda,t})=\frac{1}{2}t^{2}\lambda\|\Delta u\|_{2}^{2}-\frac{\mu}{q}t^{\frac{N(q-2)}{4}}\lambda^{\frac{q}{2}}\|u\|_{q}^{q}-\frac{1}{p}t^{\frac{N(p-2)}{4}}\lambda^{\frac{p}{2}}\|u\|_{p}^{p}$$
and
$$\mathcal{M}(\lambda,t)=Q_{p,q}(u_{\lambda,t})=t^{2}\lambda\|\Delta u\|_{2}^{2}-\mu\gamma_{q}t^{\frac{N(q-2)}{4}}\lambda^{\frac{q}{2}}\|u\|_{q}^{q}-\gamma_{p}t^{\frac{N(p-2)}{4}}\lambda^{\frac{p}{2}}\|u\|_{p}^{p}.$$
By suitable calculations, we have
$$\frac{\partial\mathcal{K}}{\partial\lambda}(1,1)=\frac{1}{2}\|\Delta u\|_{2}^{2}-\frac{\mu}{2}\|u\|_{q}^{q}-\frac{1}{2}\|u\|_{p}^{p}=-\frac{1}{2}\omega c,$$
$$\frac{\partial\mathcal{K}}{\partial t}(1,1)=\|\Delta u\|_{2}^{2}-\mu\gamma_{q}\|u\|_{q}^{q}-\gamma_{p}\|u\|_{p}^{p}=0,$$
and
$$\frac{\partial^{2}\mathcal{K}}{\partial t^{2}}(1,1)=\|\Delta u\|_{2}^{2}-\mu\gamma_{q}(\frac{N(q-2)}{4}-1)\|u\|_{q}^{q}-\gamma_{p}(\frac{N(p-2)}{4}-1)\|u\|_{p}^{p}<0,$$
which yields for $\delta_{t}$ small enough and $\delta_{\lambda}>0$,
\begin{equation}\label{eq9912}
\mathcal{K}(1+\delta_{\lambda},1+\delta_{t})<\mathcal{K}(1,1)\quad\mathrm{for}\ \omega>0.	
\end{equation}
In addition, we can also observe that
$$\mathcal{M}(1,1)=Q_{p,q}(u)=\|\Delta u\|_{2}^{2}-\mu\gamma_{q}\|u\|_{q}^{q}-\gamma_{p}\|u\|_{p}^{p}=0$$
and
$$\frac{\partial\mathcal{M}}{\partial t}(1,1)=2\|\Delta u\|_{2}^{2}-\mu\gamma_{q}\frac{N(q-2)}{4}\|u\|_{q}^{q}-\gamma_{p}\frac{N(p-2)}{4}\|u\|_{p}^{p}\not=0.$$
From the Implicit Function Theorem, we deduce that there exists $\epsilon>0$ and a continuous function $g:[1-\epsilon,\ 1+\epsilon]\mapsto\mathbb{R}$ satisfying $g(1)=1$ such that $\mathcal{M}(\lambda,g(\lambda))=0$ for $\lambda\in[1-\epsilon,\ 1+\epsilon]$. Therefore, we inter from \eqref{eq9912} that
$$m_{p,q}((1+\epsilon)c)=\inf\limits_{u\in\mathcal{Q}_{p,q}((1+\epsilon)c)}E_{p,q}(u)\leq E_{p,q}(u_{1+\epsilon,g(1+\epsilon)})<E(u)=m_{p,q}(c).$$
This proves the last assertion of the lemma.
\end{proof}

\section{Ground States In Sobolev Subcritical Case}\label{ground states1}
In this section we will establish the existence of normalized ground states to \eqref{eq02} in Sobolev subcritical case. Our prime aim is to prove the existence of a Palais-Smale sequence $\{u_{n}\}\subset\mathcal{Q}_{p,q}(c)$ at level $m_{p,q}(c)$ for $E_{p,q}$ restricted on $S(c)$.

First, let us introduce the functional $\Psi(u):H^{2}(\mathbb{R}^{N})\backslash\{0\}\rightarrow\mathbb{R}$ defined by
$$\Psi(u):=E_{p,q}(u_{t_{u}})=\frac{1}{2}t_{u}^{2}\|\Delta u\|_{2}^{2}-\frac{\mu}{q}t_{u}^{\frac{N(q-2)}{4}}\|u\|_{q}^{q}-\frac{1}{p}t_{u}^{\frac{N(p-2)}{4}}\|u\|_{p}^{p},$$
where $t_{u}$ is the unique number determined in Lemma \ref{uni}. Then we get
\begin{lemma}\label{cc1}
The functional $\Psi(u):H^{2}(\mathbb{R}^{N})\backslash\{0\}\rightarrow\mathbb{R}$ is of class $C^1$ and
\begin{align*}
	d\Psi(u)[\varphi]&=t_{u}^{2}\int_{\mathbb{R}^{N}}\Delta u\cdot\Delta\varphi dx-t_{u}^{\frac{N(q-2)}{4}}\mu\int_{\mathbb{R}^{N}}|u|^{q-2}u\cdot\varphi dx-t_{u}^{\frac{N(p-2)}{4}}\int_{\mathbb{R}^{N}}|u|^{p-2}u\cdot\varphi dx\\
	&=dE_{p,q}(u_{s_{u}})[\varphi_{s_{u}}]\\
\end{align*}
for any $u\in H^{2}(\mathbb{R}^{N})\backslash\{0\}$	 and $\varphi\in H^{2}(\mathbb{R}^{N})$.
\end{lemma}
\begin{proof}
Let  $u\in H^{2}(\mathbb{R}^{N})\backslash\{0\}$	 and $\varphi\in H^{2}(\mathbb{R}^{N})$. By the definition of $\Psi$, we have
$$\Psi(u+s\varphi)-\Psi(u)=E_{p,q}((u+s\varphi)_{t_{s}})-E_{p,q}(u_{t_{0}}),$$
where $|s|$ is small enough, $t_{s}=t_{u+s\varphi}$ and $t_{0}=t_{u}$ is the unique maximum point of the functional $E_{p,q}(u)$. By the mean value theorem we get
\begin{align*}
E_{p,q}((u+s\varphi)_{t_{s}})-E_{p,q}(u_{t_{0}})	&\leq E_{p,q}((u+s\varphi)_{t_{s}})-E_{p,q}(u_{t_{s}})\\
&=\frac{t_{s}^{2}}{2}(\|\Delta(u+s\varphi)\|_{2}^{2}-\|\Delta u\|_{2}^{2})
-\frac{t_{s}^{\frac{N(p-2)}{4}}}{p}(\|u+s\varphi\|_{p}^{p}-\|u\|_{p}^{p})\\
&-\mu\frac{t_{s}^{\frac{N(q-2)}{4}}}{q}(\|u+s\varphi\|_{q}^{q}-\|u\|_{q}^{q})\\
&=\frac{t_{s}^{2}}{2}(\int_{\mathbb{R}^{N}}2s\Delta u\cdot\Delta\varphi+s^{2}|\Delta\varphi|^{2}dx)\\
&-\mu t_{s}^{\frac{N(q-2)}{4}}\int_{\mathbb{R}^{N}}(\int_{0}^{1}|u+s\eta_{s}\varphi|^{q-2}(u+s\eta_{s}\varphi)s\varphi d\eta_{s})dx\\	
&-t_{s}^{\frac{N(p-2)}{4}}\int_{\mathbb{R}^{N}}(\int_{0}^{1}|u+s\eta_{s}\varphi|^{p-2}(u+s\eta_{s}\varphi)s\varphi d\eta_{s})dx,\\
\end{align*}
where $\eta_{s}\in(0,1)$. Similarly,
\begin{align*}
E_{p,q}((u+s\varphi)_{t_{s}})-E_{p,q}(u_{t_{0}})	&\geq E_{p,q}((u+s\varphi)_{t_{0}})-E_{p,q}(u_{t_{0}})\\
&=\frac{t_{0}^{2}}{2}(\int_{\mathbb{R}^{N}}2s\Delta u\cdot\Delta\varphi+s^{2}|\Delta\varphi|^{2}dx)\\
&-\mu t_{0}^{\frac{N(q-2)}{4}}\int_{\mathbb{R}^{N}}(\int_{0}^{1}|u+s\tau_{s}\varphi|^{q-2}(u+s\tau_{s}\varphi)s\varphi d\tau_{s})dx\\	
&-t_{0}^{\frac{N(p-2)}{4}}\int_{\mathbb{R}^{N}}(\int_{0}^{1}|u+s\tau_{s}\varphi|^{p-2}(u+s\tau_{s}\varphi)s\varphi d\tau_{s})dx,\\
\end{align*}
where $\tau_{s}\in(0,1)$. Since the map $u\mapsto t_{u}$ is of class $C^1$, it follows from the two inequalities above that
$$\lim\limits_{s\rightarrow0}\frac{\Psi(u+s\varphi)-\Psi(u)}{s}=t_{u}^{2}\int_{\mathbb{R}^{N}}\Delta u\cdot\Delta\varphi dx-t_{u}^{\frac{N(q-2)}{4}}\mu\int_{\mathbb{R}^{N}}|u|^{q-2}u\cdot\varphi dx-t_{u}^{\frac{N(p-2)}{4}}\mu\int_{\mathbb{R}^{N}}|u|^{p-2}u\cdot\varphi dx.$$
Hence, we see that the G$\hat{a}$teaux derivative of $\Psi$ is bounded linear in $\varphi$ and continuous in $u$. Therefore, $\Psi$ is of class $C^1$. In particular, by changing variables in the integrals, we have
\begin{align*}
d\Psi(u)[\varphi]&=t_{u}^{2}\int_{\mathbb{R}^{N}}\Delta u\cdot\Delta\varphi dx-t_{u}^{\frac{N(q-2)}{4}}\mu\int_{\mathbb{R}^{N}}|u|^{q-2}u\cdot\varphi dx-t_{u}^{\frac{N(p-2)}{4}}\int_{\mathbb{R}^{N}}|u|^{p-2}u\cdot\varphi dx\\
&=\int_{\mathbb{R}^{N}}\Delta u_{t_{u}}\cdot\Delta\varphi_{t_{u}}dx-\mu\int_{\mathbb{R}^{N}}|u_{t_{u}}|^{q-2}u_{t_{u}}\cdot\varphi_{t_{u}}dx-\int_{\mathbb{R}^{N}}|u_{t_{u}}|^{p-2}u_{t_{u}}\cdot\varphi_{t_{u}}dx.\\
\end{align*}
This completes the proof.
\end{proof}

For any given $c>0$ and $2+\frac{8}{N}\leq q<p<4^*$, If $q=2+\frac{8}{N}$, we further assume that $\mu c^{\frac{4}{N}}<\frac{N+4}{NC^q_{N,q}}$. We consider the constrained functional
$I:=\Psi|_{S(c)}:S(c)\rightarrow\mathbb{R}$.
\begin{lemma}\label{cc2}
The functional $I$ is of class $C^1$ and
$$dI(u)[\varphi]=d\Psi(u)[\varphi]=dE_{p,q}(u_{t_{u}})[\varphi_{t_{u}}]$$
for any $u\in S(c)$ and $\varphi\in T_{u}S(c)$.
\end{lemma}

Then we turn to search the existence of a Palais-Smale sequence $\{u_{n}\}\subset\mathcal{Q}_{p,q}(c)$ at level $m_{p,q}(c)$ for $E_{p,q}$ restricted on $S(c)$. The following arguments are inspired from \cite{bar,bar2,bar3}. We recall the below definition from \cite{GHON} first.
\begin{definition}
Let $B$ be a closed subset of a metric space $X$. We say that a class $\mathcal{G}$ of compact subsets of $Y$ is a homotopy stable family with closed boundary $B$ provided
\begin{description}
	\item $(i)$ every set in $\mathcal{G}$ contains $B$.
	\item $(ii)$ for any $A\in\mathcal{G}$ and any $\eta\in C([0,1]\times Y,Y)$ satisfying $\eta(t,x)=x$ for all $(t,x)\in(\{0\}\times Y)\bigcup([0,1],B)$, we have $\eta(\{1\}\times A)\in\mathcal{G}$.
\end{description}
We remark that $B=\emptyset$ is admissible.
\end{definition}
\begin{lemma}\label{ps1}
 Let $2+\frac{8}{N}\leq q<p<4^*$. If $q=2+\frac{8}{N}$, we further assume that $\mu c^{\frac{4}{N}}<\frac{N+4}{NC^q_{N,q}}$. Let $\mathcal{G}$ be a homotopy stable family of compact subsets of $S(c)$ (with $B=\emptyset$) and set
 $$c_{\mathcal{G}}:=\inf\limits_{A\in\mathcal{G}}\max\limits_{u\in A}I(u).$$	
 If $c_{\mathcal{G}}>0$, then there exists a Palais-Smale sequence $\{u_{n}\}\subset\mathcal{Q}_{p,q}(c)$ for the constrained functional $E_{p,q}|_{S(c)}$ at level $c_{\mathcal{G}}$.
\end{lemma}
\begin{proof}
Let $\{D_{n}\}\subset\mathcal{G}$ be an arbitrary minimizing sequence of $c_{\mathcal{G}}$. We define the mapping
$$\eta:[0,1]\times S(c)\rightarrow S(c),\qquad\eta(s,u)=u_{1-s+st_{u}},$$
which is continuous by Lemma \ref{rem}.
Since $t_{u}=1$ for any $u\in\mathcal{Q}_{p,q}(c)$ and $\emptyset=B\subset\mathcal{Q}_{p,q}(c)$, we have $\eta(s,u)=u$ for any $(s,u)\in (\{0\}\times S(c))$. Thus, by the definition of $\mathcal{G}$, one has
 \begin{equation}
 	A_{n}:=\eta(1,D_{n})=\{u_{t_{u}}:u\in D_{n}\}\in\mathcal{G}.
 \end{equation}	
 In particular, $A_{n}\subset\mathcal{G}$ for all $n\in\mathbb{N}^+$. Let $v\in A_{n}$, $i.e.\ v=u_{t_{u}}$ for some $u\in D_{n}$ and $I(v)=E_{p,q}(v_{t_{v}})=E_{p,q}(v)=E_{p,q}(u_{t_{u}})=I(u)$. It is clear that $\max_{A_{n}}I=\max_{D_{n}}I\rightarrow c_{\mathcal{G}}$. Therefore $\{A_{n}\}\subset\mathcal{Q}_{p,q}(c)$ is another minimizing sequence of $c_{\mathcal{G}}$. Now, using the equivalent minimax principle \cite{GHON}, we obtain a Palais-Smale sequence $\{v_{n}\}\subset S(c)$ for $I$ at the level $c_{\mathcal{G}}$ such that $dist_{H^{2}(\mathbb{R}^{N})}(v_{n},A_{n})\rightarrow0$ as $n\rightarrow\infty$. Let
 $$t_{n}=t_{v_{n}}\quad\mathrm{and}\quad u_{n}=(v_{n})_{t_{n}}=(v_{n})_{t_{v_{n}}}.$$
 Then we claim that there exists a $C>0$ such that,
 \begin{equation}
 t_{n}^{-2}\leq C
 \end{equation}
 for $n\in\mathbb{N}$ large enough. First, we notice that
 $$t_{n}^{-2}=\frac{\int_{H^{2}(\mathbb{R}^{N})}|\Delta v_{n}|^{2}dx}{\int_{H^{2}(\mathbb{R}^{N})}|\Delta u_{n}|^{2}dx}.$$
Since $E_{p,q}(u_{n})=I(v_{n})\rightarrow c_{\mathcal{G}}$, by Lemma \ref{uni} and Lemma \ref{cpq}, we can deduce that $|\Delta u_{n}\|_{2}$ is bounded from below by a positive constant. On the other hand, since $\{A_{n}\}\subset\mathcal{Q}_{p,q}(c)$ is a minimizing sequence for $c_{\mathcal{G}}$ and $E_{p,q}$ is coercive on $\mathcal{Q}_{p,q}(c)$, we conclude that $\{A_{n}\}$ is uniformly bounded in $H^{2}(\mathbb{R}^{N})$ and thus from $dist_{H^{2}(\mathbb{R}^{N})}(v_{n},A_{n})\rightarrow0$ as $n\rightarrow\infty$, it follows that $\sup_{n}\|v_{n}\|<\infty$. Clearly, this proves the claim.

Now, for $\{u_{n}\}\subset\mathcal{Q}_{p,q}(c)$, it follows that
$$E_{p,q}(u_{n})=I(u_{n})=I(v_{n})\rightarrow c_{\mathcal{G}}.$$
We then only need to show that $\{u_{n}\}$ is a Palais-Smale sequence for $E_{p,q}$ on $S(c)$ at level $c_{\mathcal{G}}$. For any $\phi\in T_{u_{n}}S(c)$, we have
$$\int_{\mathbb{R}^{N}}v_{n}\phi_{t_{n}^{-1}}dx=\int_{\mathbb{R}^{N}}(v_{n})_{t_{n}}\phi dx=\int_{\mathbb{R}^{N}}u_{n}\phi dx,$$
which means $\phi_{t_{n}^{-1}}\in T_{v_{n}}S(c)$. Also, by the claim it follows that
$$\|\phi_{t_{n}^{-1}}\|_{H^{2}(\mathbb{R}^{N})}\leq\max\{\sqrt{C},1\}\|\phi\|_{H^{2}(\mathbb{R}^{N})}.$$
Denoting by $\|\cdot\|_{\ast}$ the dual norm of $(T_{u}S(c))^{\ast}$ and using Lemma \ref{cc2}, we deduce that
\begin{align*}
\|dE_{p,q}(u_{n})\|_{\ast}&=\sup\limits_{\phi\in T_{u_{n}}S(c),\|\phi\|_{H^{2}(\mathbb{R}^{N})}\leq1}|dE_{p,q}(u_{n})[\phi]|\\
&=\sup\limits_{\phi\in T_{u_{n}}S(c),\|\phi\|_{H^{2}(\mathbb{R}^{N})}\leq1}|dE_{p,q}((v_{n})_{t_{n}})[(\phi_{t_{n}^{-1}})_{t_{n}}]|\\
&=\sup\limits_{\phi\in T_{u_{n}}S(c),\|\phi\|_{H^{2}(\mathbb{R}^{N})}\leq1}|dI(v_{n})[\phi_{t_{n}^{-1}}]|\\
&\leq\|dI(v_{n})\|_{\ast}\cdot\sup\limits_{\phi\in T_{u_{n}}S(c),\|\phi\|_{H^{2}(\mathbb{R}^{N})}\leq1}\|\phi_{t_{n}^{-1}}\|_{H^{2}(\mathbb{R}^{N})}\\
&\leq\max\{\sqrt{C},1\}\|dI(v_{n})\|_{\ast}.\\
\end{align*}
Since $v_{n}\subset S(c)$ is a Palais-Smale sequence of $I$, it follows that $\|dI(v_{n})\|_{\ast}\rightarrow0$ as $n\rightarrow\infty$. Hence, we deduce that $\{u_{n}\}\subset\mathcal{Q}_{p,q}(c)$ is a Palais-Smale sequence for $E_{p,q}$ on $S(c)$ at the level $c_{\mathcal{G}}$.
\end{proof}

Finally, we obtain the following lemma, which is crucial in the proof of Theorem \ref{Thm1.1}.
\begin{lemma}\label{pse}
 Let $2+\frac{8}{N}\leq q<p<4^*$, If $q=2+\frac{8}{N}$, we further assume that $\mu c^{\frac{4}{N}}<\frac{N+4}{NC^q_{N,q}}$. There exists a Palais-Smale sequence $\{u_{n}\}\subset\mathcal{Q}_{p,q}(c)$ at level $m_{p,q}(c)$ for $E_{p,q}$ restricted on $S(c)$.
\end{lemma}
\begin{proof}
	We use Lemma \ref{ps1} in the particular case where $\mathcal{G}$ is the class of all singletons belonging to $S(c)$ and then we denote it by $\mathcal{\overline{G}}$. Since $m_{p,q}(c)>0$, we only need to prove that $m_{p,q}(c)=c_{\mathcal{\overline{G}}}$. We notice that
	$$c_{\mathcal{\overline{G}}}=\inf\limits_{A\in\mathcal{\overline{G}}}\max\limits_{u\in A}I(u)=\inf\limits_{u\in S(c)}\sup\limits_{t>0}E_{p,q}(u_{t})=\inf\limits_{u\in S(c)}E_{p,q}(u_{t_{u}}).$$
For any $u\in S(c)$, we deduce from $u_{t_{u}}\in\mathcal{Q}_{p,q}(c)$ that $E_{p,q}(u_{t_{u}})\geq m_{p,q}(c)$. Therefore $c_{\mathcal{\overline{G}}}\geq m_{p,q}(c)$. On the other hand, for any $u\in\mathcal{Q}_{p,q}(c)$, we have $t_{u}=1$ and thus $E_{p,q}(u)=E_{p,q}(u_{1})\geq c_{\mathcal{\overline{G}}}$. This clearly implies that $m_{p,q}(c)\geq c_{\mathcal{\overline{G}}}$.
The proof is complete.
\end{proof}

\begin{lemma}\label{wcs}
 Let $2+\frac{8}{N}\leq q<p<4^*$. If $q=2+\frac{8}{N}$, we further assume that $\mu c^{\frac{4}{N}}<\frac{N+4}{NC^q_{N,q}}$. Let $\{u_{n}\}\subset\mathcal{Q}_{p,q}(c)$ be a Palais-Smale sequence for $E_{p,q}$ restricted on $S(c)$ at level $m_{p,q}(c)$. Then there exists $u\in S(c)$ and $\omega>0$ such that, up to translation and up to the extraction of a subsequence, $u_{n}\rightarrow u$ strongly in $H^{2}(\mathbb{R}^{N})$ and $\Delta^{2}u+\omega u-\mu|u|^{q-2}u-|u|^{p-2}u=0$.	
\end{lemma}
\begin{proof}
By Lemma \ref{cpq} $(iv)$, we can assume without loss of generality that $\{u_{n}\}\subset\mathcal{Q}_{p,q}(c)$ is a bounded Palais-Smale sequence in $H^{2}(\mathbb{R}^{N})$. After a suitable translation in $\mathbb{R}^{N}$ and up to the extraction of a subsequence, we  claim that $\{u_{n}\}$ is non-vanishing. Indeed, if not, then applying Lemma \ref{pll}, we infer that
$$0=\lim\limits_{n\rightarrow\infty}\int_{\mathbb{R}^{N}}|u_{n}|^{p}dx=\lim\limits_{n\rightarrow\infty}\int_{\mathbb{R}^{N}}|u_{n}|^{q}dx.$$
It follows that
$$\|\Delta u_{n}\|_{2}^{2}=\mu\gamma_{q}\|u_{n}\|_{q}^{q}+\gamma_{p}\|u_{n}\|_{p}^{p}\rightarrow0\quad\mathrm{as}\quad n\rightarrow\infty,$$
which contradicts the fact $\inf\limits_{u_{n}\in\mathcal{Q}_{p,q}(c)}\|\Delta u_{n}\|_{2}^{2}>0$ in Lemma \ref{cpq} $(ii)$.
Since  $\{u_{n}\}\subset\mathcal{Q}_{p,q}(c)$ is bounded in $H^{2}(\mathbb{R}^{N})$, from the condition that $\|dE_{p,q}(u_{n})\|_{\ast}\rightarrow0$ and Lemma $3$ in \cite{Hbe2}, it follows that
$$\Delta^{2}u_{n}+\omega_{n}u_{n}-\mu_{n}|u_{n}|^{q-2}u_{n}-|u_{n}|^{p-2}u_{n}\rightarrow0\quad\mathrm{in}\quad (H^{2}(\mathbb{R}^{N}))^*,$$
where
$$-c\omega_{n}=(\int_{\mathbb{R}^{N}}|\Delta u_{n}|^{2}dx-\mu\int_{\mathbb{R}^{N}}|u_{n}|^{q}dx-\int_{\mathbb{R}^{N}}|u_{n}|^{p}dx).$$
Noting that $\omega_{n}\rightarrow\omega$ for some $\omega\in\mathbb{R}$, we have
\begin{equation}\label{eq9801}
\Delta^{2}u_{n}(\cdot+y_{n})+\omega u_{n}(\cdot+y_{n})-\mu_{n}|u_{n}(\cdot+y_{n})|^{q-2}u_{n}(\cdot+y_{n})-|u_{n}(\cdot+y_{n})|^{p-2}u_{n}(\cdot+y_{n})\rightarrow0\quad\mathrm{in}\quad (H^{2}(\mathbb{R}^{N}))^*	
\end{equation}
for some $\{y_{n}\}\subset\mathbb{R}^{N}$. Since $\{u_{n}\}$ is non-vanishing, up to a subsequence, there exists $\{y_{n}^{1}\}\subset\mathbb{R}^{N}$ and $z^{1}$ such that $u_{n}(\cdot+y_{n}^{1})\rightharpoonup z^{1}$ in $H^{2}(\mathbb{R}^{N})$, $u_{n}(\cdot+y_{n}^{1})\rightarrow z^{1}$ in $L_{loc}^{r}(\mathbb{R}^{N})$ for all $r\in(2,4^{*})$ and $u_{n}(\cdot+y_{n}^{1})\rightarrow z^{1}$ almost everywhere in $\mathbb{R}^{N}$, where $z^{1}\in B(c)=\{u\in H^{2}(\mathbb{R}^{N}),\|u\|_{2}^{2}\leq c\}$.
In view of \eqref{eq9801}, we obtain
\begin{equation}\label{eq9802}
\Delta^{2}z^{1}+\omega z^{1}-\mu_{n}|z^{1}|^{q-2}z^{1}-|z^{1}|^{p-2}z^{1}=0.
\end{equation}
By the Nehari and Pohozaev identity corresponding to \eqref{eq9802}, we get $Q_{p,q}(z^{1})=0$. Let $v_{n}^{1}:=u_{n}-z^{1}(\cdot-y_{n}^{1})$
for every $n\in\mathbb{N}^+$. Clearly, $v_{n}^{1}(\cdot+y_{n})\rightharpoonup0$ in $H^{2}(\mathbb{R}^{N})$. Thus it follows that
\begin{equation}\label{eq9803}
\|u_{n}\|_{2}^{2}=\|v_{n}^{1}(\cdot+y_{n})+z^{1}\|_{2}^{2}=	\|v_{n}^{1}\|_{2}^{2}+\|z^{1}\|_{2}^{2}+o_{n}(1)\textcolor[rgb]{0.7,0.04,0.8}{,}
\end{equation}

\begin{equation}\label{eq9804}
\|\Delta u_{n}\|_{2}^{2}=\|\Delta v_{n}^{1}(\cdot+y_{n})+\Delta z^{1}\|_{2}^{2}=	\|\Delta v_{n}^{1}\|_{2}^{2}+\|\Delta z^{1}\|_{2}^{2}+o_{n}(1).
\end{equation}
By Brezis-Lieb lemma, we get
$$\|u_{n}\|_{p}^{p}=\|v_{n}^{1}\|_{p}^{p}+\|z^{1}\|_{p}^{p}+o_{n}(1)\quad\mathrm{and}\quad\|u_{n}\|_{q}^{q}=\|v_{n}^{1}\|_{q}^{q}+\|z^{1}\|_{q}^{q}+o_{n}(1).$$
It follows that
\begin{equation}\label{eq9805}
m_{p,q}(c)=\lim\limits_{n\rightarrow\infty}E_{p,q}(u_{n})=\lim\limits_{n\rightarrow\infty}E_{p,q}(u_{n}(\cdot+y_{n}))=\lim\limits_{n\rightarrow\infty}E_{p,q}(v_{n}^{1})+E_{p,q}(z^{1}).
\end{equation}
We claim that $\lim\limits_{n\rightarrow\infty}E_{p,q}(v_{n}^{1})\geq0$. Indeed,
since $Q_{p,q}(u_{n})=0$ and $Q_{p,q}(z^{1})=0$, it then follows that
$\lim\limits_{n\rightarrow\infty}Q_{p,q}(v_{n}^{1})=0$. Then, for any $2+\frac{8}{N}\leq q<p<4^*$, we have
$$\lim\limits_{n\rightarrow\infty}E_{p,q}(v_{n}^{1})=\lim\limits_{n\rightarrow\infty}(E_{p,q}(v_{n}^{1})-\frac{1}{2}Q_{p,q}(v_{n}^{1}))=\lim\limits_{n\rightarrow\infty}\big[\mu(\frac{\gamma_{q}}{2}-\frac{1}{q})\|v_{n}^{1}\|_{q}^{q}+(\frac{\gamma_{p}}{2}-\frac{1}{p})\|v_{n}^{1}\|_{p}^{p}\big]\geq0.$$
Now we set $c_{1}:=\|z^{1}\|_{2}^{2}\in(0,c]$. Since $\lim\limits_{n\rightarrow\infty}E_{p,q}(v_{n}^{1})\geq0$ and $z^{1}\in\mathcal{Q}_{p,q}(c_{1})$, it follows from \eqref{eq9805} that
$$m_{p,q}(c)=\lim\limits_{n\rightarrow\infty}E_{p,q}(v_{n}^{1})+E_{p,q}(z^{1})\geq E_{p,q}(z^{1})\geq m_{p,q}(c_{1}).$$
Recalling that $m_{p,q}(c)$ is non-increasing in $c>0$, then we have
\begin{equation}\label{eq9806}
m_{p,q}(c)=E_{p,q}(z^{1})=m_{p,q}(c_{1})	
\end{equation}
and
\begin{equation}\label{eq9807}
	\lim\limits_{n\rightarrow\infty}E_{p,q}(v_{n}^{1})=0.
\end{equation}

At last, we need to show that $c_{1}=c$. First similar in Lemma \ref{dea}, by \eqref{eq9802} and the fact $Q_{p,q}(z^1)=0$, we get $\omega>0$. If $c_{1}<c$, taking into account \eqref{eq9802}, \eqref{eq9806} and Lemma \ref{dea}, we would have
$$E_{p,q}(z^{1})=m_{p,q}(c_{1})>m_{p,q}(c),$$
which contradicts \eqref{eq9806}. Therefore, $c_{1}=\|z^{1}\|_{2}^{2}=c$ and $\|v_{n}^{1}\|_{2}^{2}\rightarrow0$. By Lemma \ref{pll}, we have
$$0=\lim\limits_{n\rightarrow\infty}\int_{\mathbb{R}^{N}}|v_{n}^{1}|^{p}dx=\lim\limits_{n\rightarrow\infty}\int_{\mathbb{R}^{N}}|v_{n}^{1}|^{q}dx,$$
it follows that
$$\lim\limits_{n\rightarrow\infty}\int_{\mathbb{R}^{N}}|\Delta v_{n}^{1}|^{2}dx=0.$$
Thus, $u_{n}(\cdot+y_{n}^{1})\rightarrow z^{1}$ strongly in $H^{2}(\mathbb{R}^{N})$. At this point, this lemma is proved.
\end{proof}

\begin{pot}
By Lemma \ref{cpq} and Lemma \ref{pse}, we have a bounded Palais-Smale sequence $\{u_{n}\}\subset\mathcal{Q}_{p,q}(c)$ for the constrained functional $E_{p,q}|_{S(c)}$ at the level $m_{p,q}(c)>0$. And then Lemma \ref{wcs} applies and provides the existence of a ground state $u\in S(c)$ at the level $m_{p,q}(c)$.
\end{pot}

\section{Ground States In Sobolev Critical Case}\label{ground states2}

In this section we will establish the existence of normalized ground states in Sobolev critical case. The main difficuly and challenge is to recover the compactness. To do this, we need to give a upper bound for $m_{4^*,q}(c)$.

\begin{lemma} \label{ub}
	Let $N\geq5$, $c>0$, $\mu>0$ and $2+\frac{8}{N}\leq q < 4^*$. If $q=2+\frac{8}{N}$, we further assume that $\mu c^{\frac{4}{N}}<\frac{N+4}{NC^q_{N,q}}$.  If $N = 5$, we further assume that $q > \frac{22}{3}$. Then $m_{4^*,q}(c) < \frac{2}{N}\mathcal{S}^{\frac{N}{4}}$, where $\mathcal{S}$ is the best constant in the Sobolev embedding $H^{2}(\mathbb{R}^{N})\hookrightarrow L^{4^*}(\mathbb{R}^{N})$.
\end{lemma}

\begin{proof}
	Define $v_\epsilon(x) = \left(\sqrt[4]{c^{-1}\|u_{\epsilon}\|_2^2}\right) ^{\frac{N-4}{2}}u_{\epsilon}(\sqrt[4]{c^{-1}\|u_{\epsilon}\|_2^2}x)$, where $u_{\epsilon}$ is given by Appendix. Then
	$$
	\int_{\mathbb{R}^N}|v_\epsilon|^2dx = c, \int_{\mathbb{R}^N}|\Delta v_\epsilon|^2dx = \int_{\mathbb{R}^N}|\Delta u_\epsilon|^2dx, \int_{\mathbb{R}^N}|v_\epsilon|^{4^*}dx = \int_{\mathbb{R}^N}|u_\epsilon|^{4^*}dx,
	$$
	and for $q \in [2+\frac{8}{N}, 4^*)$,
	\begin{eqnarray}
	\int_{\mathbb{R}^N}|v_\epsilon|^{q}dx &=& \left(\sqrt[4]{c^{-1}\|u_{\epsilon}\|_2^2}\right) ^{\frac{N-4}{2}q-N}\int_{\mathbb{R}^N}|u_\epsilon|^{q}dx \nonumber \\
	&\geq& K_1c^{\frac{N}{4}-\frac{N-4}{8}q}
	\left\{
	\begin{array}{cc}
	1, & N > 8, \\
	|\log \epsilon|^{\frac{q-4}{2}}, & N = 8, \\
	\epsilon^{\frac{8-N}{4}(N-\frac{N-4}{2}q)}, & N=6,7, \\
	\epsilon^{\frac{3}{4}(N-\frac{N-4}{2}q)}, & N=5, q > 5, \\
	\epsilon^{\frac{15}{8}}|\log \epsilon|, & N=5, q = 5, \\
	\epsilon^{\frac{5}{8}q-\frac{5}{4}}, & N=5, \frac{18}{5} \leq q < 5,
	\end{array}
	\right.
	\end{eqnarray}
	for some $K_1 > 0$ and $\epsilon$ small enough, where the last inequality follows by Lemma \ref{lem App.1}. Next we use $v_\epsilon$ to estimate $m_{4^*,q}(c)$. By direct calculations, one has
	\begin{eqnarray} \label{eq 4.2}
	&& E_{4^*,q}((v_\epsilon)_{t_\epsilon}) \nonumber \\
	&=& \frac{1}{2}t_\epsilon^{2}\|\Delta v_\epsilon\|_{2}^{2}-\frac{\mu}{q}t_\epsilon^{\frac{N(q-2)}{4}}\|v_\epsilon\|_{q}^{q}-\frac{1}{4^*}t_\epsilon^{\frac{N(4^*-2)}{4}}\|v_\epsilon\|_{4^*}^{4^*} \nonumber \\
	&\leq& \frac{1}{2}t_\epsilon^{2}\left(\mathcal{S}^{\frac{N}{4}}+O(\epsilon^{N-4})\right) -\frac{1}{4^*}t_\epsilon^{4^*}\left(\mathcal{S}^{\frac{N}{4}}+O(\epsilon^{N})\right) \nonumber \\
	&-& \frac{\mu}{q}t_\epsilon^{\frac{N(q-2)}{4}}K_1c^{\frac{N}{4}-\frac{N-4}{8}q}
	\left\{
	\begin{array}{cc}
	1, & N > 8, \\
	|\log \epsilon|^{\frac{q-4}{2}}, & N = 8, \\
	\epsilon^{\frac{8-N}{4}(N-\frac{N-4}{2}q)}, & N=6,7, \\
	\epsilon^{\frac{3}{4}(N-\frac{N-4}{2}q)}, & N=5, q > 5, \\
	\epsilon^{\frac{15}{8}}|\log \epsilon|, & N=5, q = 5, \\
	\epsilon^{\frac{5}{8}q-\frac{5}{4}}, & N=5, \frac{18}{5} \leq q < 5,
	\end{array}
	\right.
	\end{eqnarray}
	where $t_\epsilon = t_{v_\epsilon}$ is given by Lemma \ref{uni} and $(v_\epsilon)_{t_\epsilon} \in\mathcal{Q}_{4^*,q}(c)$. The definition of $m_{4^*,q}(c)$ yields that $m_{4^*,q}(c) \leq E_{4^*,q}((v_\epsilon)_{t_\epsilon})$.

    We claim that there exist $t_0, t_1 > 0$ independent of $\epsilon$ such that $t_\epsilon \in [t_0,t_1]$ for $\epsilon > 0$ small. Suppose by contradiction that $t_\epsilon \rightarrow 0$ or $t_\epsilon \rightarrow \infty$ as $\epsilon \rightarrow 0$. The fact that $(v_\epsilon)_{t_\epsilon} \in\mathcal{Q}_{4^*,q}(c)$ yields that
    \begin{equation} \label{eq 4.3}
    \|\Delta v_\epsilon\|_{2}^{2} = \mu\gamma_{q}t_\epsilon^{\frac{N(q-2)}{4}-2}\|v_\epsilon\|_{q}^{q} + \gamma_{4^*}t_\epsilon^{\frac{N(4^*-2)}{4}-2}\|v_\epsilon\|_{4^*}^{4^*}.
    \end{equation}
    When $q > 2 + \frac{8}{N}$, the left side of \eqref{eq 4.3} is $\mathcal{S}^{\frac{N}{4}}+O(\epsilon^{N-4}) > 0$ for small $\epsilon$ while the right side of \eqref{eq 4.3} tends to $0$ as $t_\epsilon \rightarrow 0$ and tends to infinity as $t_\epsilon \rightarrow \infty$, which is impossible. When $q = 2 + \frac{8}{N}$ and $t_\epsilon \rightarrow \infty$, we find a similar contradiction since the right side of \eqref{eq 4.3} tends to infinity. Furthermore, when $q = 2 + \frac{8}{N}$,  Gagliardo-Nirenberg inequality shows that
    \begin{equation}
    	\|v_\epsilon\|_{q}^{q}\leq C^{q}_{N,q}\|\Delta v_\epsilon\|_{2}^{2}c^{\frac{4}{N}}.
    \end{equation}
    We further assume that $\mu c^{\frac{4}{N}}<\frac{N+4}{NC^q_{N,q}}$. Therefore,
    \begin{equation}
    \frac{N}{N+4}\mu\|v_\epsilon\|_{q}^{q} < \|\Delta v_\epsilon\|_{2}^{2} = \mathcal{S}^{\frac{N}{4}}+O(\epsilon^{N-4}).
    \end{equation}
    The above inequality, together with \eqref{eq 4.3}, enables us to show that it is impossible that $t_\epsilon \rightarrow 0$. Thus, the claim holds.

    In \eqref{eq 4.2}, when $N = 5$, $q > \frac{22}{3}$, or $N \geq 6$, $O(\epsilon^{N-4})$ can be controlled by the last term for $\epsilon > 0$ small enough. Hence,
    $$
    m_{4^*,q}(c) \leq E_{4^*,q}((v_\epsilon)_{t_\epsilon}) < \sup_{t \geq 0}\left( \frac{1}{2}t^2 - \frac{1}{4^*}t^{4^*}\right)\mathcal{S}^{\frac{N}{4}} = \frac{2}{N}\mathcal{S}^{\frac{N}{4}}.
    $$
    The proof is complete.
\end{proof}

\begin{proof}[Proof of Theorem \ref{Thm1.2}]
	Similarly to the discussions in Section \ref{ground states1}, we can obtain a bounded Palais-Smale sequence (in $H^{2}(\mathbb{R}^{N})$) $\{u_{n}\}\subset{\mathcal{Q}_{4^*,q}(c)}$ for the constrained functional $E_{4^*,q}|_{S(c)}$ at the level $m_{4^*,q}(c)$. By Lemma \ref{ub} and Lemma \ref{cpq}, $0 < m_{4^*,q}(c) < \frac{2}{N}\mathcal{S}^{\frac{N}{4}}$. After a suitable translation in $\mathbb{R}^{N}$ and up to the extraction of a subsequence, we  claim that $\{u_{n}\}$ is non-vanishing. Indeed, if not, then applying Lemma \ref{pll}, we infer that
    $$
    \lim\limits_{n\rightarrow\infty}\int_{\mathbb{R}^{N}}|u_{n}|^{q}dx = 0.
    $$
    Since $u_{n}$ is bounded in $H^{2}(\mathbb{R}^{N})$, passing to a subsequence if necessary, we may assume that
    $$
    \lim_{n \rightarrow \infty}\|\Delta u_{n}\|_{2}^2 = \lim_{n \rightarrow \infty}\|u_{n}\|_{4^*}^{4^*} = m \in [0,\infty).
    $$
    By Lemma \ref{cpq}, $m > 0$. Then, by the Sobolev inequality, $\mathcal{S}m^{\frac{2}{4^*}} \leq m$, implying that $m \geq \mathcal{S}^{\frac{4^*}{4^*-2}} = \mathcal{S}^{\frac{N}{4}}$. Therefore, we derive a self-contradictory inequality
    $$
    \frac{2}{N}\mathcal{S}^{\frac{N}{4}} > m_{4^*,q}(c) = \lim_{n \rightarrow \infty}E_{4^*,q}(u_n) = \frac{2}{N}m \geq \frac{2}{N}\mathcal{S}^{\frac{N}{4}}.
    $$
    Thus, the claim holds.

    Up to a subsequence, there exists $\{y_{n}^{1}\}\subset\mathbb{R}^{N}$ and $z^{1}$ such that $u_{n}(\cdot+y_{n}^{1})\rightharpoonup z^{1} \neq 0$ in $H^{2}(\mathbb{R}^{N})$, $u_{n}(\cdot+y_{n}^{1})\rightarrow z^{1}$ in $L_{loc}^{r}(\mathbb{R}^{N})$ for all $r\in(2,4^{*})$ and $u_{n}(\cdot+y_{n}^{1})\rightarrow z^{1}$ almost everywhere in $\mathbb{R}^{N}$, where $z^{1}\in B(c)=\{u\in H^{2}(\mathbb{R}^{N}), 0 < \|u\|_{2}^{2}\leq c\}$. Let $v_{n}^{1}:=u_{n}-z^{1}(\cdot-y_{n}^{1})$ for every $n\in\mathbb{N}^+$. Mimicking the proof of Lemma \ref{wcs}, we can prove that $\|z^{1}\|_{2}^{2}=c$, $Q_{4^*,q}(v_{n}^{1})\rightarrow0$, $E_{4^*,q}(v_{n}^{1})\rightarrow0$ and $\|v_{n}^{1}\|_{2}^{2}\rightarrow0$. By Lemma \ref{pll}, we have
    $$
    \lim\limits_{n\rightarrow\infty}\int_{\mathbb{R}^{N}}|v_{n}^{1}|^{q}dx = 0.
    $$
    Then,
    $$
    \|\Delta v_{n}^{1}\|_{2}^2 - \|v_{n}^{1}\|_{4^*}^{4^*} = Q_{4^*,q}(v_{n}^{1}) + \mu\gamma_{q}\|v_{n}^{1}\|_{q}^{q} \rightarrow0,
    $$
    $$
    \frac{1}{2}\|\Delta v_{n}^{1}\|_{2}^2 - \frac{1}{4^*}\|v_{n}^{1}\|_{4^*}^{4^*} = E_{4^*,q}(v_{n}^{1}) + \frac{\mu}{q}\gamma_{q}\|v_{n}^{1}\|_{q}^{q} \rightarrow0,
    $$
    imply that $\|\Delta v_{n}^{1}\|_{2}^2 \rightarrow0, \|v_{n}^{1}\|_{4^*}^{4^*}\rightarrow0$. Thus, $u_{n}(\cdot+y_{n}^{1})\rightarrow z^{1}$ strongly in $H^{2}(\mathbb{R}^{N})$, and we complete the proof.
\end{proof}

\section{Non-existence result when $q = 2+\frac{8}{N}$, $\mu c^{\frac{4}{N}} \geq \frac{N+4}{NC^q_{N,q}}$} \label{non-existence}

In this section, we write $E_{p,q}^\mu, Q_{p,q}^\mu, \mathcal{Q}_{p,q}^\mu(c), m_{p,q}^\mu(c)$ in place of $E_{p,q}, Q_{p,q}, \mathcal{Q}_{p,q}(c), m_{p,q}(c)$.

\begin{lemma}\label{uni2}
	Let $N\geq5$, $c>0$, $\mu>0$ and $2+\frac{8}{N}= q<p\leq4^*$. We further assume that $\mu c^{\frac{4}{N}} \geq \frac{N+4}{NC^q_{N,q}}$.
	\begin{description}
		\item $(i)\ $ For any $u\in S(c)$ such that $\frac{\|\Delta u\|_{2}^2}{\|u\|_{q}^q} > \frac{N}{N+4}\mu$, there exists a unique $t_{u}\in(0,+\infty)$ for which $u_{t_{u}}\in\mathcal{Q}_{p,q}^\mu(c)$. Moreover, $t_{u}$ is the unique critical point of $E_{p,q}^\mu(u_{t})$ such that $E_{p,q}^\mu(u_{t_{u}})=\max_{t\in(0,+\infty)}E_{p,q}^\mu(u_{t}).$
		\item $(ii)\ $ For any $u\in S(c)$ such that $\frac{\|\Delta u\|_{2}^2}{\|u\|_{q}^q} \leq \frac{N}{N+4}\mu$, $u_{t}\notin\mathcal{Q}_{p,q}^\mu(c), \forall t \in (0,\infty)$. Moreover, $\mathcal{Q}_{p,q}^\mu(c) = \{u_{t_{u}}: u\in S(c), \frac{\|\Delta u\|_{2}^2}{\|u\|_{q}^q} > \frac{N}{N+4}\mu\}.$
		\item $(iii)\ \mathcal{Q}_{p,q}^\mu(c)\not=\emptyset$ if $5 \leq N \leq 8$.
		\item $(iv)\ $ If $\mathcal{Q}_{p,q}^\mu(c)\not=\emptyset$, then $m_{p,q}^\mu(c)=0$.
	\end{description}	
\end{lemma}
\begin{proof}
	First we let
	$$
	Q_{p,q}^\mu(u_{t})=t^{2}\left( \|\Delta u\|_{2}^{2}-\frac{N}{N+4}\mu\|u\|_{q}^{q} - \gamma_{p}t^{\frac{N(p-2)}{4}-2}\|u\|_{p}^{p}\right).
	$$
	When $\frac{\|\Delta u\|_{2}^2}{\|u\|_{q}^q} > \frac{N}{N+4}\mu$, we complete the proof of $(i)$ similar to the one of Lemma \ref{uni}. When $\frac{\|\Delta u\|_{2}^2}{\|u\|_{q}^q} \leq \frac{N}{N+4}\mu$, it is easy to see that $Q_{p,q}^\mu(u_{t}) < 0$ for all $t > 0$, and we  complete the proof of $(ii)$.
	
	When $5 \leq N \leq 8$, for all $\mu > 0$, we can always choose $u \in S(c)$ such that $\frac{\|\Delta u\|_{2}^2}{\|u\|_{q}^q} > \frac{N}{N+4}\mu$. Indeed, we consider $v_\epsilon$ given in the proof of Lemma \ref{ub}. The proof of Lemma \ref{ub} shows that $\|v_{\epsilon}\|_{2}^{2} = c$, $\|v_{\epsilon}\|_{q}^{q} \rightarrow 0$ if $5 \leq N \leq 8$ and $\|\Delta v_{\epsilon}\|_{2}^{2} \rightarrow \mathcal{S}^{\frac{N}{4}}$ as $\epsilon \rightarrow 0$. Thus $v_\epsilon$ satisfies our condition for $\epsilon$ small enough. 
	Then we know that $\mathcal{Q}_{p,q}^\mu(c)\not=\emptyset$.\vspace{2em}

\vspace{2em}Finally, for $\mu > 0$ such that $\mathcal{Q}_{p,q}^\mu(c)\not=\emptyset$, we can choose $u \in S(c)$ such that $\frac{\|\Delta u\|_{2}^2}{\|u\|_{q}^q} > \frac{N}{N+4}\mu$. Noticing that
	\begin{equation}
	E_{p,q}^\mu(u)= (\frac{\gamma_{p}}{2}-\frac{1}{p})\|u\|_{p}^{p}, \forall u \in \mathcal{Q}_{p,q}^\mu(c),
	\end{equation}
	we obtain that $m_{p,q}^\mu(c)\geq0$ immediately. We will show that $m_{p,q}^\mu(c)\leq0$ to complete the proof.

	Claim 1: There exists $\{w_n\} \subset S(c)$ such that $w_n \rightarrow w_\infty$ in $H^2$, $\frac{\|\Delta w_n\|_{2}^2}{\|w_n\|_{q}^q} > \frac{N}{N+4}\mu$ and $\frac{\|\Delta w_\infty\|_{2}^2}{\|w_\infty\|_{q}^q} = \frac{N}{N+4}\mu$.
	
	By the proof of $(iii)$, we can take $\phi_1 \in S(c)$ such that $\frac{\|\Delta \phi_1\|_{2}^2}{\|\phi_1\|_{q}^q} > \frac{N}{N+4}\mu$. Furthermore, we can take a minimizer $\phi_2$ of the Gagliardo-Nirenberg inequality \eqref{eq05} such that $\|\phi_2\|_{2}^2 = c$. Since $\mu c^{\frac{4}{N}} \geq \frac{N+4}{NC^q_{N,q}}$, we have
	$$
	\|\Delta \phi_2\|_{2}^2 = \frac{1}{C_{N,q}^qc^{\frac{4}{N}}}\|\phi_2\|_{q}^q \leq \frac{N}{N+4}\mu\|\phi_2\|_{q}^q.
	$$
	Obviously, $\phi_1 \neq -\phi_2$. Hence, we can consider
	$$
	\rho_\tau = \sqrt{c}\frac{\tau\phi_1 + (1-\tau)\phi_2}{\|\tau\phi_1 + (1-\tau)\phi_2\|_2} \in S_c, \tau \in [0,1].
	$$
	By the continuity of $\rho_\tau$ in $H^2$, there exists $\tau_0 \in [0,1)$ such that $\|\Delta \rho_{\tau_0}\|_{2}^2 = \frac{N}{N+4}\mu\|\rho_{\tau_0}\|_{q}^q$, $\|\Delta \rho_{\tau}\|_{2}^2 > \frac{N}{N+4}\mu\|\rho_{\tau}\|_{q}^q$ for all $\tau > \tau_0$. Take $w_\infty = \rho_{\tau_0}$, $w_n = \rho_{\tau_n}$ with $\tau_n > \tau_0$, $\tau_n \rightarrow \tau_0$. Thus the Claim 1 is proved.
	
	Claim 2: $t_n = t_{w_n} \rightarrow 0$ where $t_{w_n}$ is given by $(i)$.
	
	By the proof of $(i)$, we have
	$$
	t_n^{\frac{N(p-2)}{4}-2} = \frac{\|\Delta w_n\|_{2}^{2}-\frac{N}{N+4}\mu\|w_n\|_{q}^{q}}{\gamma_{p}\|w_n \|_{p}^{p}} \rightarrow \frac{\|\Delta w_\infty\|_{2}^{2}-\frac{N}{N+4}\mu\|w_\infty\|_{q}^{q}}{\gamma_{p}\|w_\infty\|_{p}^{p}} = 0.
	$$
	Since $\frac{N(p-2)}{4}-2 > 0$, we complete the proof of Claim 2.
	
	Claims 1, 2 yield that $E_{p,q}^\mu((w_n)_{t_n}) \rightarrow 0$. Noticing that $(w_n)_{t_n} \in \mathcal{Q}_{p,q}^\mu(c)$, we get
	$$
	m_{p,q}^\mu(c) = \inf_{u \in \mathcal{Q}_{p,q}^\mu(c)}E_{p,q}^\mu(u) \leq \lim_{n \rightarrow \infty}E_{p,q}^\mu((w_n)_{t_n}) = 0.
	$$
	The proof is complete.
\end{proof}	

\begin{proof}[Proof of Theorem \ref{Thm1.3}]
	Suppose the contrary that for $\mu c^{\frac{4}{N}} \geq \frac{N+4}{NC^q_{N,q}}$, $m_{p,q}^\mu(c)$ is attained by some $u_{\mu,c}$. By Lemma \ref{uni2} $(ii)$, either $u_{\mu,c} \in \mathcal{Q}_{p,q}^\mu(c)$ or $\|\Delta u_{\mu,c}\|_{2}^2=\frac{N}{N+4}\mu\|u_{\mu,c}\|_{q}^q$, $\| u_{\mu,c}\|_{2}^2 = c$. If $u_{\mu,c} \in \mathcal{Q}_{p,q}^\mu(c)$, then
	$$
	0 = m_{p,q}^\mu(c) = E_{p,q}^\mu(u_{\mu,c}) = (\frac{\gamma_{p}}{q\gamma_{q}}-\frac{1}{p})\|u_{\mu,c}\|_{p}^{p},
	$$
	in a contradiction with $u_{\mu,c} \in S_c$. If $\|\Delta u_{\mu,c}\|_{2}^2=\frac{N}{N+4}\mu\|u_{\mu,c}\|_{q}^q$, then
    $$
    0 = m_{p,q}^\mu(c) = E_{p,q}^\mu(u_{\mu,c}) = -\frac{1}{p}\|u_{p,q}\|_{p}^{p},
    $$
    which is impossible since $\| u_{\mu,c}\|_{2}^2 = c$.
\end{proof}
\color{black}

\section{Instability} \label{insta}
In this section, we consider the instability of standing waves associated with the ground states. We will consider our problem in $H^{2}_{rad}(\mathbb{R}^{N})$. In fact, we need to use the local virial identity to prove the instability by blowup and it is not clear how to control errors term without radial symmetry.

\begin{remark}
 For now, we haven't removed the radial symmetry assumption. The strategy here was introduced in \cite{OGA}, where Ogawa and Tsutsumi showed blowup for radial solutions for mass supercritical NLS with radial $u_{0}\in H^1(\mathbb{R}^{N})$ with infinite variance ($i.e.$ we may not have $xu_{0}\not\in L^2(\mathbb{R}^{N})$). We requires a careful analysis of the time evolution for the localized virial identity $M_{\varphi_{R}}[u(t)]$,  which is essential in the proof of blowup result.
\end{remark}
 Firstly, from \cite{PAU} we recall the local well-posedness of the Cauchy problem \eqref{eq01} holds for $2+\frac{8}{N}\leq q<p\leq4^*$.

\begin{proposition}\label{lwp}
Let $2+\frac{8}{N}\leq q<p\leq4^*$. For any $u_{0}\in H^{2}(\mathbb{R}^{N})$, there exists $T>0$ and a unique solution $u(t)\in C([0,T);H^{2}(\mathbb{R}^{N}))$ to \eqref{eq01}, which satisfies the conservation laws along the time, that is, for any $t\in[0,T)$
$$\|u(t)\|_{2}=\|u_{0}\|_{2},\ E_{p,q}(u(t))=E_{p,q}(u_{0}).$$
Moreover, for the Sobolev subcritical case, we have the following blowup alternative: either $T=\infty$, or $\lim_{t\rightarrow T}\|\Delta u(t)\|_{2}=\infty$. For the Sobolev critical case $p=4^*$, we have a blowup alternative that involves a critical Strichartz norm in space time, $i.e.$ it holds that $0<T<\infty$ and
\begin{equation}\label{eq01060}
\int_{0}^{T}\int_{\mathbb{R}^{N}}|u(t,x)|^{\frac{2(N+4)}{N-4}}dxdt=+\infty.	
\end{equation}
\end{proposition}

Next we recall the localized virial to \eqref{eq01} introduced in \cite{BOU}, which will play an important role to deduce the occurrence of blowup. Let $\varphi:\mathbb{R}^{N}\rightarrow\mathbb{R}$ be a radial function such that $D^i(\varphi)\in L^{\infty}(\mathbb{R}^{N})$, $1\leq i\leq6$,
\[\varphi:=
\begin{cases}
   \frac{r^{2}}{2}, \qquad\qquad r\leq1\\
    const\qquad\quad r\geq10,\\
    \end{cases}\]
and $\varphi''(r)\leq1$ for all $r\geq0$. We set $\varphi_{R}(r)=R^2\varphi(\frac{r}{R})$ for $R>0$. Then we record some properties of $\varphi_{R}$ which will be often used in the following.
\begin{lemma}\label{bzy}
For all $r\geq0$, we have
\begin{description}
\item $(i)\ 1-\varphi''_{R}(r)\geq0,\quad1-\frac{\varphi'_{R}(r)}{r}\geq0,$
\item $(ii)\ N-\Delta\varphi_{R}(r)\geq0.$
\end{description}
\end{lemma}
\begin{proof}
$(i)\ $ By the definition of $\varphi_{R}$, we have $\varphi_{R}'(r)=R\varphi'(\frac{r}{R})$ and $\varphi_{R}''(r)=\varphi''(\frac{r}{R})$. Since $\varphi''(r)\leq1$ for $r\geq0$, we get $\varphi_{R}''(r)\leq1$. Then integrating differential inequality $\varphi_{R}''\leq1$ on $[0,r]$,
 we have $\int_{0}^{r}\varphi_{R}''(s)ds\leq r$. It follows that $\varphi_{R}'(r)\leq r$, then we get our conclusion.

$(ii)\ $ By simple calculations, we can get $\Delta\varphi_{R}(r)=\varphi_{R}''(r)+\varphi_{R}'(r)\frac{N-1}{r}$. Hence, according to $(i)$, it follows that
\begin{align*}
N-\Delta\varphi_{R}(r)&=N-\varphi_{R}''(r)-\varphi_{R}'(r)\frac{N-1}{r}\\
&=1-\varphi_{R}''(r)+N-1-\varphi_{R}'(r)\frac{N-1}{r}\\
&=1-\varphi_{R}''(r)+(N-1)(1-\frac{\varphi_{R}'(r)}{r})\geq0\\
\end{align*}
 \end{proof}

\begin{lemma}\label{bzy2}
We have
\begin{description}
\item $(i)$ \[\nabla\varphi_{R}(r)=R\varphi'_{R}(\frac{r}{R})\frac{x}{|x|}=
\begin{cases}
x,\qquad r\leq R\\
0,\qquad r\geq10R,\\	
\end{cases}\]
\item $(ii)$
$$\|\nabla^{j}\varphi_{R}\|_{L^{\infty}}\leq R^{2-j},\quad\mathrm{for}\ 0\leq j\leq6,$$
\item $(iii)$
\[\mathrm{supp}(\nabla^{j}\varphi_{R})\subset
\begin{cases}
\{|x|\leq10R\},\qquad\qquad\mathrm{for}\ j=1,2\\
\{R\leq|x|\leq10R\},\qquad\mathrm{for}\ 3\leq j\leq6.\\	
\end{cases}\]
\end{description}
\end{lemma}
\begin{proof}
These properties can be checked easily from the definition of $\varphi$.
\end{proof}
For $u\in H^{2}(\mathbb{R}^{N})$, we define the localized virial $M_{\varphi_{R}}$ of $u$ by
\begin{equation}\label{eq6001}
M_{\varphi_{R}}[u]:=\big<u,-i(\nabla\varphi_{R}\cdot\nabla+\nabla\cdot\nabla\varphi_{R})u\big>=2Im\int_{\mathbb{R}^{N}}\overline{u}\nabla\varphi_{R}\cdot\nabla udx	.
\end{equation}
By Cauchy-Schwartz inequality and H\"older inequality, we have $|M_{\varphi_{R}}[u]|\lesssim
R\|u\|^{\frac{3}{2}}_{2}\|\Delta u\|^{\frac{1}{2}}_{2}$. In particular, the localized virial is well-defined in $H^{2}(\mathbb{R}^{N})$.

The following lemma reveals a key point on the evolution of this quantity.
\begin{lemma}\label{vir}
Let $2+\frac{8}{N}\leq q<p\leq4^*$ and $R>0$. Suppose that $u\in C([0,T), H^{2}_{rad}(\mathbb{R}^{N}))$ is a solution to \eqref{eq01}. Then for any $t\in[0,T)$, we have the following inequality:
\begin{align*}
\frac{d}{dt}M_{\varphi_{R}}[u(t)]&\leq2N(q-2)E_{p,q}(u)-(N(q-2)-8)\int_{\mathbb{R}^{N}}|\Delta u(t)|^{2}-\frac{2N(p-q)}{p}\int_{\mathbb{R}^{N}}|u(t)|^{p}\\
&+O(R^{-4}+R^{-2}\|\Delta u(t)\|_{2}+R^{-\frac{(q-2)(N-1)}{2}}\|\Delta u\|_{2}^{\frac{q-2}{4}}+R^{-\frac{(p-2)(N-1)}{2}}\|\Delta u\|_{2}^{\frac{p-2}{4}})\\
&=8Q_{p,q}(u(t))+O(R^{-4}+R^{-2}\|\Delta u(t)\|_{2}+R^{-\frac{(q-2)(N-1)}{2}}\|\Delta u\|_{2}^{\frac{q-2}{4}}+R^{-\frac{(p-2)(N-1)}{2}}\|\Delta u\|_{2}^{\frac{p-2}{4}}). \\	
\end{align*}
\end{lemma}

\begin{proof}
First, we rewrite $M_{\varphi_{R}}[u]$ as
$$M_{\varphi_{R}}[u]=\big<u(t),\Gamma_{\varphi_{R}}u(t)\big>\quad\mathrm{with}\ \Gamma_{\varphi_{R}}=-i(\nabla\varphi_{R}\cdot\nabla+\nabla\cdot\nabla\varphi_{R}).$$	
Recall that if u satisfies the equation $i\partial_{t}u=Hu$, then
$$\frac{d}{dt}\big<u,Au\big>=\big<u,[H,iA]u\big>,$$
where $[H,A]=HA-AH$ is a commutator operator. Then by taking the time derivative and using that $i\partial_{t}u$ is given in \eqref{eq01}, we get that
\begin{equation}\label{eq1301}
\frac{d}{dt}M_{\varphi_{R}}[u(t)]=A_{R}^{1}[u(t)]+B_{R}^{2}[u(t)]+B_{R}^{3}[u(t)]	
\end{equation}
with
$$A_{R}^{1}[u]=:\big<u(t),[\Delta^{2}u+i\Gamma_{\varphi_{R}}]u(t)\big>,$$
$$B_{R}^{2}[u]=:\big<u(t),[-\mu|u|^{q-2}u+i\Gamma_{\varphi_{R}}]u(t)\big>\quad\mathrm{and}\quad B_{R}^{3}[u]=:\big<u(t),[-|u|^{p-2}u+i\Gamma_{\varphi_{R}}]u(t)\big>.$$
For convenience, we state some identities to be used in the next. By using the fact that $\Delta A+A\Delta=2\partial_{k}A\partial_{k}+[\partial_{k},[\partial_{k},A]]$, we observe that
\begin{equation}\label{eq1302}
[\Delta^{2},i\Gamma_{\varphi_{R}}]=\Delta[\Delta,i\Gamma_{\varphi_{R}}]+[\Delta,i\Gamma_{\varphi_{R}}]\Delta=2\partial_{k}[\Delta,i\Gamma_{\varphi_{R}}]\partial_{k}+[\partial_{k},[\partial_{k},[\Delta,i\Gamma_{\varphi_{R}}]]].
\end{equation}
Then by simple calculations, it follows that
\begin{equation}\label{eq1303}
[\Delta,i\Gamma_{\varphi_{R}}]=[\Delta,\nabla\varphi_{R}\cdot\nabla+\nabla\cdot\nabla\varphi_{R}]=4\partial_{k}(\partial^{2}_{kl}\varphi_{R})\partial_{l}+\Delta^{2}\varphi_{R}.
\end{equation}
Combining \eqref{eq1302} and \eqref{eq1303}, we obtain the identity
\begin{equation}\label{eq1304}
[\Delta^{2},i\Gamma_{\varphi_{R}}]=8	\partial^{2}_{kl}(\partial^{2}_{lm}\varphi_{R})\partial^{2}_{mk}+4\partial_{k}(\partial^{2}_{kl}\Delta\varphi_{R})\partial_{l}+2\partial_{k}(\Delta^{2}\varphi_{R})\partial_{k}+\Delta^{3}\varphi_{R}.
\end{equation}
We recall that the Hessian of sufficiently regular and radial function $f:\mathbb{R}^N\rightarrow\mathbb{C}$ is given by
\begin{equation}\label{eq1305}
\partial^{2}_{kl}f=(\delta_{kl}-\frac{x_{k}x_{l}}{r^2})\frac{\partial_{r}f}{r}+\frac{x_{k}x_{l}}{r^2}\partial^{2}_{r}f.
\end{equation}
Applying \eqref{eq1305} to $\varphi_{R}(r)$ and $u(t,r)$, a calculation combined with integration by parts yields that
\begin{align*}
	8\big<u,8\partial^{2}_{kl}(\partial^{2}_{lm}\varphi_{R})\partial^{2}_{mk}u\big>&=8\int_{\mathbb{R}^{N}}(\partial^{2}_{kl}\overline{u})(\partial^{2}_{lm}\varphi_{R})(\partial^{2}_{mk}u)\\
	&=8\int_{\mathbb{R}^{N}}(\partial_{r}^{2}\varphi_{R}|\partial_{r}^{2}u|^{2}+\frac{d-1}{r^{2}}\frac{\partial_{r}\varphi_{R}}{r}|\partial_{r}u|^{2})\\
	&=8\int_{\mathbb{R}^{N}}|\Delta u|^{2}-(1-\partial_{r}^{2}\varphi_{R})|\partial_{r}^{2}u|^{2}-(1-\frac{\partial_{r}\varphi_{R}}{r})\frac{d-1}{r^{2}}|\partial_{r}u|^{2}.\\
\end{align*}
Noticing that $\int_{\mathbb{R}^{N}}|\Delta u|^{2}=\int_{\mathbb{R}^{N}}|\partial_{r}^{2}u|^{2}+\frac{d-1}{r^{2}}|\partial_{r}u|^{2}$ for radial  $u\in H^2(\mathbb{R}^{N})$, according to the inequalities in Lemma \ref{bzy},  we have
\begin{equation}\label{eq1306}
8\big<u,8\partial^{2}_{kl}(\partial^{2}_{lm}\varphi_{R})\partial^{2}_{mk}u\big>\leq8\int_{\mathbb{R}^{N}}|\Delta u|^{2}.
\end{equation}
For the rest terms in \eqref{eq1304}, direct arguments yield that
\begin{eqnarray}\label{eq1307}
\begin{array}{cc}
&|\big<u,\partial_{k}(\partial^{2}_{kl}\Delta\varphi_{R})\partial_{l}u\big>|\lesssim\|\partial^{2}_{kl}\Delta\varphi_{R}\|_{L^\infty}\|\nabla u\|^{2}_{2}\lesssim R^{-2}\|\nabla u\|^{2}_{2}\lesssim R^{-2}\|u\|_{2}\|\Delta u\|_{2},\\
&|\big<u,\partial_{k}(\Delta^{2}\varphi_{R})\partial_{k}u\big>|\lesssim\|\Delta^{2}\varphi_{R}\|_{L^\infty}\|\nabla u\|^{2}_{2}\lesssim R^{-2}\|\nabla u\|^{2}_{2}\lesssim R^{-2}\|u\|_{2}\|\Delta u\|_{2},\\
&|\big<u,\Delta^{3}\varphi_{R}u\big>||\lesssim\|\Delta^{3}\varphi_{R}\|_{L^\infty}\|u\|^{2}_{2}\lesssim R^{-4}\|u\|^{2}_{2}.\\
	\end{array}	
\end{eqnarray}
By \eqref{eq1306} and \eqref{eq1307}, we conclude that
\begin{equation}\label{eq1308}
A_{R}^{1}[u(t)]\leq8\int_{\mathbb{R}^{N}}|\Delta u(t)|^{2}+O(R^{-4}+R^{-2}\|\Delta u(t)\|_{2}).	
\end{equation}
Next, we turn to the combined nonlinearity. By using the identity $\nabla(|u|^{q})=\frac{q}{q-2}\nabla(|u|^{q-2})|u|^{2}$, we note that integration by parts yields that
\begin{align*}
	B_{R}^{2}[u]&=:-\big<u,[\mu|u|^{q-2}u,\nabla\varphi_{R}\cdot\nabla+\nabla\cdot\nabla\varphi_{R}]u\big>=2\mu\int_{\mathbb{R}^{N}}|u|^{2}\nabla\varphi_{R}\nabla(|u|^{q-2})\\
	&=-\mu\frac{2(q-2)}{q}\int_{\mathbb{R}^{N}}\Delta \varphi_{R}|u|^{q}.\\
\end{align*}
By the definition of $\varphi_{R}(r)$, we get that $\Delta\varphi_{R}(r)-N\equiv0$ for $r\leq R$, it follows that
\begin{align*}
B_{R}^{2}[u]&=\mu\frac{-2N(q-2)}{q}\int_{\mathbb{R}^{N}}|u|^{q}-\mu\frac{2(q-2)}{q}\int_{|x|\geq R}(\Delta \varphi_{R}-N)|u|^{q}\\	
&=\mu\frac{-2N(q-2)}{q}\int_{\mathbb{R}^{N}}|u|^{q}+O(R^{-\frac{(q-2)(N-1)}{2}}\|\Delta u\|_{2}^{\frac{q-2}{4}}),
\end{align*}
where the last term follows from the fact $\|\Delta\varphi_{R}(r)-N\|_{L^{\infty}}\lesssim1$ and applying the Strauss inequality, which gives us
$$\int_{|x|\geq R}|u|^{q}\lesssim\|u\|_{2}^{2}\|u\|_{L^{\infty}}^{q-2}\lesssim R^{-\frac{(q-2)(N-1)}{2}}\|u\|_{2}^{\frac{q+2}{2}}\|\nabla u\|_{2}^{\frac{q-2}{2}}\lesssim R^{-\frac{(q-2)(N-1)}{2}}\|u\|_{2}^{\frac{3q+2}{4}}\|\Delta u\|_{2}^{\frac{q-2}{4}}.$$
Similarly, we can also obtain that
\begin{align*}
B_{R}^{3}[u]&=\frac{-2N(p-2)}{p}\int_{\mathbb{R}^{N}}|u|^{p}-\frac{2(p-2)}{p}\int_{|x|\geq R}(\Delta \varphi_{R}-N)|u|^{p}\\	
&=\frac{-2N(p-2)}{p}\int_{\mathbb{R}^{N}}|u|^{p}+O(R^{-\frac{(p-2)(N-1)}{2}}\|\Delta u\|_{2}^{\frac{p-2}{4}}).
\end{align*}
Finally, from the above, we conclude that
\begin{align*}
\frac{d}{dt}M_{\varphi_{R}}[u(t)]&\leq2N(q-2)E_{p,q}(u)-(N(q-2)-8)\int_{\mathbb{R}^{N}}|\Delta u(t)|^{2}-\frac{2N(p-q)}{p}\int_{\mathbb{R}^{N}}|u(t)|^{p}\\
&+O(R^{-4}+R^{-2}\|\Delta u(t)\|_{2}+R^{-\frac{(q-2)(N-1)}{2}}\|\Delta u\|_{2}^{\frac{q-2}{4}}+R^{-\frac{(p-2)(N-1)}{2}}\|\Delta u\|_{2}^{\frac{p-2}{4}})\\
&=8Q_{p,q}(u(t))+O(R^{-4}+R^{-2}\|\Delta u(t)\|_{2}+R^{-\frac{(q-2)(N-1)}{2}}\|\Delta u\|_{2}^{\frac{q-2}{4}}+R^{-\frac{(p-2)(N-1)}{2}}\|\Delta u\|_{2}^{\frac{p-2}{4}}).
\end{align*}
This completes the proof.
\end{proof}
\subsection{Sobolev Subcritical case: $2+\frac{8}{N}\leq q<p<4^*$}

\begin{proof}[Proof of Theorem \ref{orbi}]
	By Theorem \ref{Thm1.1}, let $u\in S(c)$ be a ground state solution to \eqref{eq02} at the level $m_{p,q}(c)>0$. Since we work in the space	 $H^{2}_{rad}(\mathbb{R}^{N})$, then $u$ is radially symmetric. By Lemma \ref{uni}, we can have $\|u_{s}-u\|_{H^2}\rightarrow0$	and $E_{p,q}(u_{s})\rightarrow m_{p,q}(c)$ as $s\rightarrow1^+$. Then we set $\psi_{0}=u_{s}$. By Proposition \ref{lwp}, we know that there exists a unique solution $\psi(t)\in C([0,T);H_{rad}^{2}(\mathbb{R}^{N}))$ to \eqref{eq01} with the initial datum $\psi_{0}=u_{s}$, where $T$ is the maximal time of existence. In view of Definition \ref{inst}, to prove that $u$ is strongly unstable, we need to show that $\psi(t)$ blows up in a finite time. We divide the proof into five steps:

\noindent Step $1$: $\mathcal{M}_{p,q}(c)$ is nonempty, where $\mathcal{M}_{p,q}(c)$ is defined as follows:
\begin{equation}\label{eq602}
\mathcal{M}_{p,q}(c)=\{v\in S(c), E_{p,q}(v)<E_{p,q}(u),Q_{p,q}(v)<0\}.	
\end{equation}

By Lemma \ref{uni} and the conservation laws, it follows that $\psi_{0}=u_{s}\in\mathcal{M}_{p,q}(c)$ for any $s>1$, which means $\mathcal{M}_{p,q}(c)\not=\emptyset$.

\noindent Step $2$: We claim that $\mathcal{M}_{p,q}(c)$ is invariant under the flow of \eqref{eq01}.

Since $\psi_{0}\in \mathcal{M}_{p,q}(c)$, by Proposition \ref{lwp}, there exist $T>0$ and a unique solution $\psi(t)\in C([0,T);H_{rad}^{2}(\mathbb{R}^{N}))$ to \eqref{eq01} with the initial datum $\psi_{0}$. First, by conservation law, we know that $E_{p,q}(\psi(t))=E_{p,q}(\psi_{0})<E_{p,q}(u)$. Then we show that $Q_{p,q}(\psi(t))<0$. If not, by the continuity, there exists a $t_{0}\in(0,T)$ such that $Q_{p,q}(\psi(t_{0}))=0$. Hence
$$m_{p,q}(c)\leq E_{p,q}(\psi(t_{0}))<E_{p,q}(u)=m_{p,q}(c),$$
which is a contradiction.

\noindent Step $3$: We claim that there exists a constant $a>0$ such that $Q_{p,q}(\psi(t))\leq-a$ for all $t\in[0,T)$.

For convenience, we shall write $\psi=\psi(t)$. Since $Q_{p,q}(\psi(t))<0$, by Lemma
\ref{uni}, there exists a $s_{\psi}\in (0,1)$ such that $Q_{p,q}(\psi_{s_{\psi}})=0$. Recalling that $s\rightarrow E_{p,q}(\psi_{s})$ is concave on $[s_{\psi},\infty)$, it follows that
$$E_{p,q}(\psi_{s_{\psi}})-E_{p,q}(\psi)\leq(s_{\psi}-1)\frac{d}{ds}E_{p,q}(\psi_{s})|_{s=1}=(s_{\psi}-1)Q_{p,q}(\psi).$$
Since $Q_{p,q}(\psi)<0$ and $m_{p,q}(c)\leq E_{p,q}(\psi_{s_{\psi}})$, by the conservation of energy, we have
$$Q_{p,q}(\psi)<(1-s_{\psi})Q_{p,q}(\psi)\leq E_{p,q}(\psi)-E_{p,q}(\psi_{s_{\psi}})\leq E_{p,q}(\psi_{0})-m_{p,q}(c):=-a<0.$$
\noindent Step $4$: Let us show that there exist a $\delta>0$ such that
\begin{equation}\label{eq603}
	\frac{d}{dt}M_{\varphi_{R}}[\psi(t)]\leq-\delta\|\Delta\psi(t)\|_{2}^{2}\quad\mathrm{for}\ t\in[0,T)
\end{equation}
and a $t_{0}>0$ such that
\begin{equation}\label{eq604}
M_{\varphi_{R}}[\psi(t)]<0,\quad\mathrm{for}\ t>t_{0}.
\end{equation}
First, we claim that there exists a constant $C$ such that
\begin{equation}
\|\Delta \psi\|_{2}^{2}\geq C,
\end{equation}
for $t\in[0,T)$. If not, there existed $t_{k}\subset[0,T)$ such that $\|\Delta \psi(t_{k})\|_{2}^{2}\rightarrow0$. By the Gagliardo-Nirenberg inequality and conservation law, we have
$$\|\psi(t_{k})\|_{q}^{q}\leq C^{q}_{N,q}\|\Delta\psi(t_{k})\|_{2}^{q\gamma_{q}}\|\psi(t_{k})\|_{2}^{q(1-\gamma_{q})}\rightarrow0$$
and
$$\|\psi(t_{k})\|_{p}^{p}\leq C^{p}_{N,p}\|\Delta\psi(t_{k})\|_{2}^{p\gamma_{q}}\|\psi(t_{k})\|_{2}^{p(1-\gamma_{p})}\rightarrow0$$
as $k\rightarrow\infty$. It follows that
$$Q_{p,q}(\psi(t_{k}))=\|\Delta\psi(t_{k})\|_{2}^{2}-\mu\gamma_{q}\|\psi(t_{k})\|_{q}^{q}-\gamma_{p}\|\psi(t_{k})\|_{p}^{p}\rightarrow0,$$
which is a contradiction to the fact $Q_{p,q}(\psi(t))\leq-a$ for all $t\in[0,T)$. Since the solution $\psi$ is radial, we recall that

\begin{align*}
\frac{d}{dt}M_{\varphi_{R}}[\psi(t)]&\leq8\|\Delta\psi(t)\|_{2}^{2}-8\mu\gamma_{q}\|\psi(t)\|_{q}^{q}-8\gamma_{p}\|\psi(t)\|_{p}^{p}\\
&+O(R^{-4}+R^{-2}\|\Delta\psi(t)\|_{2}+R^{-\frac{(q-2)(N-1)}{2}}\|\Delta \psi(t)\|_{2}^{\frac{q-2}{4}}+R^{-\frac{(p-2)(N-1)}{2}}\|\Delta\psi(t)\|_{2}^{\frac{p-2}{4}}),
\end{align*}
for all $t\in[0,T)$ and $R>1$. Then we apply the Young inequality to obtain for any $\xi>0$,

$$R^{-2}\|\Delta\psi(t)\|_{2}\leq C\xi\|\Delta\psi(t)\|^{2}_{2}+\xi^{-1}R^{-4},$$
$$R^{-\frac{(q-2)(N-1)}{2}}\|\Delta\psi(t)\|_{2}^{\frac{q-2}{4}}\leq C\xi\|\Delta \psi(t)\|^{2}_{2}+\xi^{-\frac{q-2}{10-q}}R^{-\frac{4(q-2)(N-1)}{10-q}},$$
and
$$R^{-\frac{(p-2)(N-1)}{2}}\|\Delta\psi(t)\|_{2}^{\frac{p-2}{4}}\leq C\xi\|\Delta\psi(t)\|^{2}_{2}+\xi^{-\frac{p-2}{10-p}}R^{-\frac{4(p-2)(N-1)}{10-p}}.$$
Therefore, we obtain
\begin{align*}
\frac{d}{dt}M_{\varphi_{R}}[\psi(t)]&\leq8\|\Delta\psi(t)\|_{2}^{2}-8\mu\gamma_{q}\|\psi(t)\|_{q}^{q}-8\gamma_{p}\|\psi(t)\|_{p}^{p}\\
&+C\xi\|\Delta\psi(t)\|^{2}_{2}+O(R^{-4}+\xi^{-\frac{q-2}{10-q}}R^{-\frac{4(q-2)(N-1)}{10-q}}+\xi^{-\frac{p-2}{10-p}}R^{-\frac{4(p-2)(N-1)}{10-p}}),
\end{align*}
for all $t\in[0,T)$, any $R>1$, any $\xi>0$ and some constant $C>0$. For simplicity, we denote $\mathcal{K}=R^{-4}+\xi^{-\frac{q-2}{10-q}}R^{-\frac{4(q-2)(N-1)}{10-q}}+\xi^{-\frac{p-2}{10-p}}R^{-\frac{4(p-2)(N-1)}{10-p}}$.

When $2+\frac{8}{N}<q<p<4^*$, we consider two cases. Before that, we set
\begin{equation}\label{eq2100}
\rho=\frac{4N(q-2)|E_{p,q}(\psi_{0})|+2}{N(q-2)-8}.	
\end{equation}

\noindent $Case\ 1:$ Assume that $t\in[0,T)$ such that
 \begin{equation}\label{eq1401}
\|\Delta\psi(t)\|^{2}_{2}\leq\rho.	
 \end{equation}
Then we have
$$
	\frac{d}{dt}M_{\varphi_{R}}[\psi(t)]\leq-8a+C\xi\rho+O(\mathcal{K})\quad\mathrm{for}\ t\in[0,T).
$$
Taking $\xi>0$ small enough and $R>1$ large enough depending on $\xi$, it follows that
\begin{equation}\label{eq1402}
	\frac{d}{dt}M_{\varphi_{R}}[\psi(t)]\leq-4a\leq\frac{-4a}{\rho}\|\Delta\psi(t)\|_{2}^{2}\quad\mathrm{for}\ t\in[0,T).
\end{equation}
where $a$ is determined in the last step.

\noindent $Case\ 2:$ Assume that $t\in[0,T)$ such that
 \begin{equation}\label{eq1401}
 \|\Delta\psi(t)\|^{2}_{2}>\rho.	
 \end{equation}
In this case by the conservation law, it follows that
\begin{align*}
Q_{p,q}(\psi(t))&=2N(q-2)E_{p,q}(\psi(t))-(N(q-2)-8)\|\Delta \psi(t)\|_{2}^{2}-\frac{2N(p-q)}{p}\|\psi(t)\|_{p}^{p}\\
&\leq\rho\frac{N(q-2)-8}{2}-1-\frac{N(q-2)-8}{2}\rho-\frac{N(q-2)-8}{2}\|\Delta \psi(t)\|_{2}^{2}.\\
\end{align*}
 Thus we get that
\begin{equation}\label{eq2101}
\frac{d}{dt}M_{\varphi_{R}}[\psi(t)]\leq-1-\frac{N(q-2)-8}{2}\|\Delta \psi(t)\|_{2}^{2}+C\xi\|\Delta u(t)\|^{2}_{2}+O(\mathcal{K}).	
\end{equation}
Since $2+\frac{8}{N}<q<p<4^*$ and $N\geq5$, we first choose $\xi>0$ small enough such that
$$\frac{N(q-2)-8}{2}-C\xi>\frac{N(q-2)-8}{4}.$$
Then taking $R>1$ large enough depending on $\xi$, we can get that
$$-1+O(\mathcal{K})\leq0.$$
Hence, we can obtain that
\begin{equation}\label{eq2200}
\frac{d}{dt}M_{\varphi_{R}}[\psi(t)]\leq	-\frac{N(q-2)-8}{4}\|\Delta \psi(t)\|_{2}^{2}.
\end{equation}

When $2+\frac{8}{N}=q<p<4^*$, we can still get
\begin{equation}\label{eq2000}
	\frac{d}{dt}M_{\varphi_{R}}[\psi(t)]\leq-\delta\|\Delta\psi(t)\|_{2}^{2}\quad\mathrm{for}\ t\in[0,T).
\end{equation}
In fact, we also need to distinguish into two cases. First we denote
\begin{equation}\label{eq230217}
\rho_{0}=\frac{8\beta_{2}\sigma|E_{p,q}(\psi_{0})|+C_{\eta}\frac{8\mu\sigma(\beta_{2}-\beta_{1})}{q}\|\psi_{0}\|_{2}^{2}+2}{\beta_{2}\sigma-4},	
\end{equation}
where $\beta_{1}=\frac{N(q-2)}{2}$, $\beta_{2}=\frac{N(p-2)}{2}$and $C_{\eta}$ will be chosen later. In particular, we can find a small constant $\epsilon$ such that $p>2+\frac{8+\epsilon}{N}$, it follows that $\sigma=\frac{8+\epsilon}{(p-2)N}<1$. If $\|\Delta\psi(t)\|^{2}_{2}\leq\rho_{0}$, we can get the same conclusion as above.

Now we assume that $\|\Delta\psi(t)\|^{2}_{2}>\rho_{0}$. Then by the conservation of energy, we can get
\begin{align*}
Q_{p,q}(\psi(t))&=8\|\Delta\psi(t)\|^{2}_{2}-\frac{4\mu\beta_{1}}{q}\|\psi(t)\|_{q}^{q}-\frac{4\beta_{2}}{p}\|\psi(t)\|_{p}^{p}\\
&=4\beta_{2}\sigma E_{p,q}(\psi_{0})+(8-2\beta_{2}\sigma)\|\Delta\psi(t)\|^{2}_{2}+(\frac{4\mu\beta_{2}\sigma}{q}-\frac{4\mu\beta_{1}}{q})\|\psi(t)\|_{q}^{q}+(\frac{4\beta_{2}\sigma}{p}-\frac{4\beta_{2}}{p})\|\psi(t)\|_{p}^{p}\\
&\leq4\beta_{2}\sigma E_{p,q}(\psi_{0})+(8-2\beta_{2}\sigma)\|\Delta\psi(t)\|^{2}_{2}+\frac{4\mu\sigma(\beta_{2}-\beta_{1})}{q}\|\psi(t)\|_{q}^{q}-\frac{4\beta_{2}(1-\sigma)}{p}\|\psi(t)\|_{p}^{p}.\\
\end{align*}
By the Young inequality, we can obtain that
\begin{align*}
\frac{4\mu\sigma(\beta_{2}-\beta_{1})}{q}\|\psi(t)\|_{q}^{q}&=\frac{4\mu\sigma(\beta_{2}-\beta_{1})}{q}\int_{\mathbb{R}^{N}}|\psi(t)|^{q}dx=\frac{4\mu\sigma(\beta_{2}-\beta_{1})}{q}\int_{\mathbb{R}^{N}}|\psi|^{\frac{2(p-q)}{p-2}}|\psi|^{q-\frac{2(p-q)}{p-2}}dx\\
&\leq\frac{4\mu\sigma(\beta_{2}-\beta_{1})}{q}\big[\frac{p-q}{p-2}\int_{\mathbb{R}^{N}}(\eta^{-\frac{q-\frac{2(p-q)}{p-2}}{p}}|\psi|^{\frac{2(p-q)}{p-2}})^{\frac{p-2}{p-q}}dx\\
&+\frac{q-\frac{2(p-q)}{p-2}}{p}\int_{\mathbb{R}^{N}}(\eta^{\frac{q-\frac{2(p-q)}{p-2}}{p}}|\psi|^{q-\frac{2(p-q)}{p-2}})^{\frac{p}{q-\frac{2(p-q)}{p-2}}}dx\big]\\
&\leq\frac{4\mu\sigma(\beta_{2}-\beta_{1})}{q}\big[C_{\eta}\int_{\mathbb{R}^{N}}|\psi(t)|^{2}dx+\eta|\psi(t)|^{p}dx\big].\\	
\end{align*}
Choosing $\eta$ small enough such that
$$\frac{4\mu\sigma(\beta_{2}-\beta_{1})}{q}\eta<\frac{4\beta_{2}(1-\sigma)}{p},$$
which, together with the definition of $\rho_{0}$, implies that
\begin{align*}
Q_{p,q}(\psi(t))&\leq4\beta_{2}\sigma E_{p,q}(\psi_{0})+(8-2\beta_{2}\sigma)\|\Delta\psi(t)\|^{2}_{2}+\frac{4\mu\sigma(\beta_{2}-\beta_{1})}{q}\|\psi(t)\|_{q}^{q}-\frac{4\beta_{2}(1-\sigma)}{p}\|\psi(t)\|_{p}^{p}\\	
&\leq4\beta_{2}\sigma E_{p,q}(\psi_{0})-(2\beta_{2}\sigma-8)\|\Delta\psi(t)\|^{2}_{2}+\frac{4\mu\sigma(\beta_{2}-\beta_{1})}{q}C_{\eta}\|\psi\|^{2}_{2}\\
&\leq\frac{\beta_{2}\sigma-4}{2}\rho_{0}-1-\frac{\beta_{2}\sigma-4}{2}\rho_{0}-\frac{3(\beta_{2}\sigma-4)}{2}\|\Delta\psi(t)\|^{2}_{2}.\\
\end{align*}
By a similar argument above, we can obtain that
\begin{equation}\label{eq2201}
\frac{d}{dt}M_{\varphi_{R}}[\psi(t)]\leq	-\frac{\beta_{2}\sigma-4}{2}\|\Delta \psi(t)\|_{2}^{2}.
\end{equation}	
Noticing that
$$|\frac{d}{dt}M_{\varphi_{R}}[\psi(t)]|\geq\min\{\frac{4a}{\rho},\frac{N(q-2)-8}{4},\frac{4a}{\rho_{0}},\frac{\beta_{2}\sigma-4}{2}\}\|\Delta \psi(t)\|_{2}^{2}$$
and
$$M_{\varphi_{R}}[\psi(t)]=M_{\varphi_{R}}[\psi(0)]+\int_{0}^{t}\frac{d}{ds}M_{\varphi_{R}}[\psi(s)]ds,$$
we get that there exists a $t_{0}>0$ such that for any $t>t_{0}$, \eqref{eq604} holds.

\noindent Step $5$: We now conclude that the solution $\psi(t)$ with the initial data $\psi_{0}=u_{s}$ blows up in finite time. Suppose by contradiction that $T=\infty$. Then integrating \eqref{eq603} on $[t_{0},t]$, we have that
\begin{equation}\label{eq605}
M_{\varphi_{R}}[\psi(t)]\leq-\int_{t_{0}}^{t}\delta\|\Delta \psi(t)\|_{2}^{2}dx.	
\end{equation}
Now using the Cauchy-Schwartz's inequality, we get from the definition of $M_{\varphi_{R}}[\psi(t)]$ that
\begin{equation}\label{eq606}
|M_{\varphi_{R}}[\psi(t)]|^{4}\leq C\|\Delta \psi(t)\|_{2}^{2}.	
\end{equation}
Thus combining \eqref{eq605} and \eqref{eq606}, we have for some $C>0$,
\begin{equation}\label{eq607}
M_{\varphi_{R}}[\psi(t)]\leq-C\int_{t_{0}}^{t}|M_{\varphi_{R}}[\psi(t)]|^4dx	
\end{equation}
Setting $k(t)=\int_{t_{0}}^{t}|M_{\varphi_{R}}[\psi(\tau)]|^{4}d\tau$, we get from \eqref{eq607} that
\begin{equation}\label{eq608}
k'(t)\geq Ck(t)^{4}. 	
\end{equation}
Then integrating differential inequality \eqref{eq608} on $[t_{0},t]$, we obtain that
$$M_{\varphi_{R}}[\psi(t)]\leq-Ck(t)\leq\frac{-Ck(t_{0})}{(1-3Ck^{3}(t_{0})(t-t_{0}))^{\frac{1}{3}}}.$$
This shows $M_{\varphi_{R}}[\psi(t)]\rightarrow-\infty$ as $t$ tends to some limit time $t^*$. Hence the solution $\psi(t,x)$ cannot exist globally. Recalling the blowup alternative we conclude our proof.
\end{proof}

\subsection{Sobolev critical case: $2+\frac{8}{N}\leq q<p=4^*$}

\begin{proof}[Proof of Theorem \ref{orbi2}]
	By Theorem \ref{Thm1.2}, let $u\in S(c)$ be a ground state solution to \eqref{eq02} at the level $m_{4^*,q}(c)>0$. Similar in Sobolev subcritical case, $u$ is radially symmetric since we work in the space $H^{2}_{rad}(\mathbb{R}^{N})$. By Lemma \ref{uni}, we can have $\|u_{s}-u\|_{H^2}\rightarrow0$	and $E_{4^*,q}(u_{s})\rightarrow m_{4^*,q}(c)$ as $s\rightarrow1^+$. Then we set $\psi_{0}=u_{s}$. By Proposition \ref{lwp}, we know that there exists a unique solution $\psi(t)\in C([0,T);H_{rad}^{2}(\mathbb{R}^{N}))$ to \eqref{eq01} with the initial datum $\psi_{0}=u_{s}$. According to Definition \ref{inst},  we need to show that $\psi(t)$ blows up in a finite time. As before, we also divide the proof into five steps:

\noindent Step $1$: $\mathcal{M}_{4^*,q}(c)$ is nonempty, where $\mathcal{M}_{4^*,q}(c)$ is defined as follows:
\begin{equation}\label{eq0105}
\mathcal{M}_{4^*,q}(c)=\{v\in S(c), E_{4^*,q}(v)<E_{4^*,q}(u),Q_{4^*,q}(v)<0\}.	
\end{equation}

\noindent Step $2$: $\mathcal{M}_{4^*,q}(c)$ is invariant under the flow of \eqref{eq01}.

\noindent Step $3$: There exists a constant $a^*>0$ such that $Q_{4^*,q}(\psi(t))\leq-a^*$ for all $t\in[0,T)$.

\noindent For the first three steps, we omit the proof since it turns out to be the same as that of the Sobolev subcritical case.

\noindent Step $4$: There exist a $\delta^*>0$ such that
\begin{equation}\label{eq01051}
	\frac{d}{dt}M_{\varphi_{R}}[\psi(t)]\leq-\delta^*\|\Delta\psi(t)\|_{2}^{2},\quad\mathrm{for}\ t\in[0,T)
\end{equation}
and a $t_{0}>0$ such that
\begin{equation}\label{eq01052}
M_{\varphi_{R}}[\psi(t)]<0,\quad\mathrm{for}\ t>t_{0}.
\end{equation}
First, we claim that there exists a costant $C$ such that
\begin{equation}\label{eq01053}
\|\Delta \psi\|_{2}^{2}\geq C,
\end{equation}
for $t\in[0,T)$. If not, there existed $t_{k}\subset[0,T)$ such that $\|\Delta \psi(t_{k})\|_{2}^{2}\rightarrow0$. By \eqref{eq230213}, the Gagliardo-Nirenberg inequality and conservation law, we have
$$\|\psi(t_{k})\|_{q}^{q}\leq C^{q}_{N,q}\|\Delta\psi(t_{k})\|_{2}^{q\gamma_{q}}\|\psi(t_{k})\|_{2}^{q(1-\gamma_{q})}\rightarrow0$$
and
$$\|\psi(t_{k})\|^{4^*}_{4^*}\leq\mathcal{S}^{\frac{4^*}{2}}\|\Delta\psi(t_{k})\|^{4^*}_{2}\rightarrow0$$
as $k\rightarrow\infty$. It follows that
$$Q_{4^*,q}(\psi(t_{k}))=\|\Delta\psi(t_{k})\|_{2}^{2}-\mu\gamma_{q}\|\psi(t_{k})\|_{q}^{q}-\|\psi(t_{k})\|_{4^*}^{4^*}\rightarrow0,$$
which is a contradiction to the fact $Q_{4^*,q}(\psi(t))\leq-a^*$ for $t\in[0,T)$. Then we recall that

\begin{align*}
\frac{d}{dt}M_{\varphi_{R}}[\psi(t)]&\leq8\|\Delta\psi(t)\|_{2}^{2}-8\mu\gamma_{q}\|\psi(t)\|_{q}^{q}-8\|\psi(t)\|_{4^*}^{4^*}\\
&+O(R^{-4}+R^{-2}\|\Delta\psi(t)\|_{2}+R^{-\frac{(q-2)(N-1)}{2}}\|\Delta\psi(t)\|_{2}^{\frac{q-2}{4}}+R^{-\frac{(4^*-2)(N-1)}{2}}\|\Delta\psi(t)\|_{2}^{\frac{4^*-2}{4}}),
\end{align*}
for all $t\in[0,T)$ and $R>1$. Then we apply the Young inequality to obtain for any $\xi>0$,

$$R^{-2}\|\Delta\psi(t)\|_{2}\leq C\xi\|\Delta\psi(t)\|^{2}_{2}+\xi^{-1}R^{-4},$$
$$R^{-\frac{(q-2)(N-1)}{2}}\|\Delta\psi(t)\|_{2}^{\frac{q-2}{4}}\leq C\xi\|\Delta\psi(t)\|^{2}_{2}+\xi^{-\frac{q-2}{10-q}}R^{-\frac{4(q-2)(N-1)}{10-q}},$$
and
$$R^{-\frac{(4^*-2)(N-1)}{2}}\|\Delta\psi(t)\|_{2}^{\frac{4^*-2}{4}}\leq C\xi\|\Delta\psi(t)\|^{2}_{2}+\xi^{-\frac{4^*-2}{10-4^*}}R^{-\frac{4(4^*-2)(N-1)}{10-4^*}}(N\geq6),$$
$$R^{-\frac{(4^*-2)(N-1)}{2}}\|\Delta\psi(t)\|_{2}^{\frac{4^*-2}{4}}=R^{-16}\|\Delta\psi(t)\|_{2}^{2}(N=5).$$
Therefore, we obtain
\begin{align*}
\frac{d}{dt}M_{\varphi_{R}}[\psi(t)]&\leq8\|\Delta\psi(t)\|_{2}^{2}-8\mu\gamma_{q}\|\psi(t)\|_{q}^{q}-8\|\psi(t)\|_{4^*}^{4^*}\\
&+C\xi\|\Delta\psi(t)\|^{2}_{2}+O(R^{-4}+\xi^{-\frac{q-2}{10-q}}R^{-\frac{4(q-2)(N-1)}{10-q}}+\xi^{-\frac{4^*-2}{10-4^*}}R^{-\frac{4(4^*-2)(N-1)}{10-4^*}}) (N \geq 6),
\end{align*}
\begin{align*}
	\frac{d}{dt}M_{\varphi_{R}}[\psi(t)]&\leq8\|\Delta\psi(t)\|_{2}^{2}-8\mu\gamma_{q}\|\psi(t)\|_{q}^{q}-8\|\psi(t)\|_{4^*}^{4^*}\\
	&+(C\xi+R^{-16})\|\Delta\psi(t)\|^{2}_{2}+O(R^{-4}+\xi^{-\frac{q-2}{10-q}}R^{-\frac{4(q-2)(N-1)}{10-q}})(N = 5),
\end{align*}
for all $t\in[0,T)$, any $R>1$, any $\xi>0$ and some constant $C>0$. For simplicity, we denote $\mathcal{K}=R^{-4}+\xi^{-\frac{q-2}{10-q}}R^{-\frac{4(q-2)(N-1)}{10-q}}+\xi^{-\frac{4^*-2}{10-4^*}}R^{-\frac{4(4^*-2)(N-1)}{10-4^*}}$ if $N \geq 6$ and $\mathcal{K}= R^{-4}+\xi^{-\frac{q-2}{10-q}}R^{-\frac{4(q-2)(N-1)}{10-q}}$ if $N = 5$.

When $2+\frac{8}{N}<q<p=4^*$, we consider two cases. Before that, we set
\begin{equation}\label{eq01054}
\rho^*=\frac{4N(q-2)|E_{4^*,q}(\psi_{0})|+2}{N(q-2)-8}.	
\end{equation}

\noindent $Case\ 1:$ Assume that $t\in[0,T)$ such that
 \begin{equation}\label{eq01055}
\|\Delta\psi(t)\|^{2}_{2}\leq\rho^*.	
 \end{equation}
Then we have
$$
	\frac{d}{dt}M_{\varphi_{R}}[\psi(t)]\leq-8a^*+(C\xi+R^{-16})\rho+O(\mathcal{K})\quad\mathrm{for}\ t\in[0,T).
$$
Taking $\xi>0$ small enough and $R>1$ large enough depending on $\xi$, it follows that
\begin{equation}\label{eq01056}
	\frac{d}{dt}M_{\varphi_{R}}[\psi(t)]\leq-4a^*\leq\frac{-4a^*}{\rho^*}\|\Delta\psi(t)\|_{2}^{2}\quad\mathrm{for}\ t\in[0,T),
\end{equation}
where $a^*$ is determined in Step 2.

\noindent $Case\ 2:$ Assume that $t\in[0,T)$ such that
 \begin{equation}\label{eq01057}
 \|\Delta\psi(t)\|^{2}_{2}>\rho^*.	
 \end{equation}
In this case by the conservation law, it follows that
\begin{align*}
Q_{4^*,q}(\psi(t))&=2N(q-2)E_{4^*,q}(\psi_{0})-(N(q-2)-8)\|\Delta \psi(t)\|_{2}^{2}-\frac{2N(4^*-q)}{4^*}\|\psi(t)\|_{4^*}^{4^*}\\
&\leq\rho^*\frac{N(q-2)-8}{2}-1-\frac{N(q-2)-8}{2}\rho^*-\frac{N(q-2)-8}{2}\|\Delta \psi(t)\|_{2}^{2}.\\
\end{align*}
 Thus we get that
\begin{equation}\label{eq01058}
\frac{d}{dt}M_{\varphi_{R}}[\psi(t)]\leq-1-\frac{N(q-2)-8}{2}\|\Delta \psi(t)\|_{2}^{2}+(C\xi+R^{-16})\|\Delta \psi(t)\|^{2}_{2}+O(\mathcal{K}).	
\end{equation}
Since $2+\frac{8}{N}<q<p=4^*$ and $N\geq5$, we first choose $\xi>0$ small enough and $R>1$ large enough depending on $\xi$ such that
$$C\xi + R^{-16} -\frac{N(q-2)-8}{2} < -\frac{N(q-2)-8}{4}.$$
Also, we can get that
$$-1+O(\mathcal{K})\leq0.$$
Hence, we can obtain that
\begin{equation}\label{eq01059}
\frac{d}{dt}M_{\varphi_{R}}[\psi(t)]\leq-\frac{N(q-2)-8}{4}\|\Delta \psi(t)\|_{2}^{2}.
\end{equation}

When $2+\frac{8}{N}=q<p=4^*$, we can still get
\begin{equation}\label{eq010510}
	\frac{d}{dt}M_{\varphi_{R}}[\psi(t)]\leq-\delta^*\|\Delta\psi(t)\|_{2}^{2}\quad\mathrm{for}\ t\in[0,T).
\end{equation}
In fact, we also need to distinguish into two cases. First we denote
$$\rho_{0}^*=\frac{8\beta^*_{2}\sigma^*|E_{4^*,q}(\psi_{0})|+C_{\eta}\frac{8\mu\sigma^*(\beta^*_{2}-\beta_{1})}{q}\|\psi_{0}\|_{2}^{2}+2}{\beta_{2}\sigma^*-4},$$
where $\beta^*_{2}=\frac{N(4^*-2)}{2}$and $C_{\eta}$ will be chosen later. In particular, we can find a small constant $\epsilon$ such that $4^*>2+\frac{8+\epsilon}{N}$, it follows that $\sigma^*=\frac{8+\epsilon}{(4^*-2)N}<1$. If $\|\Delta\psi(t)\|^{2}_{2}\leq\rho_{0}^*$, we can get the same conclusion as above.

Now we assume that $\|\Delta\psi(t)\|^{2}_{2}>\rho_{0}^*$. Then by the conservation of energy, we can get
\begin{align*}
Q_{4^*,q}(\psi(t))&=8\|\Delta\psi(t)\|^{2}_{2}-\frac{4\mu\beta_{1}}{q}\|\psi(t)\|_{q}^{q}-\frac{4\beta_{2}^*}{4^*}\|\psi(t)\|_{4^*}^{4^*}\\
&=4\beta_{2}^*\sigma^* E_{4^*,q}(\psi_{0})+(8-2\beta_{2}^*\sigma^*)\|\Delta\psi(t)\|^{2}_{2}+(\frac{4\mu\beta_{2}^*\sigma^*}{q}-\frac{4\mu\beta_{1}}{q})\|\psi(t)\|_{q}^{q}+(\frac{4\beta_{2}^*\sigma^*}{4^*}-\frac{4\beta_{2}^*}{4^*})\|\psi(t)\|_{4^*}^{4^*}\\
&\leq4\beta_{2}^*\sigma^* E_{4^*,q}(\psi_{0})+(8-2\beta_{2}^*\sigma^*)\|\Delta\psi(t)\|^{2}_{2}+\frac{4\mu\sigma^*(\beta_{2}^*-\beta_{1})}{q}\|\psi(t)\|_{q}^{q}-\frac{4\beta_{2}^*(1-\sigma^*)}{4^*}\|\psi(t)\|_{4^*}^{4^*}.\\
\end{align*}
By the Young inequality, we can obtain that
\begin{align*}
\frac{4\mu\sigma^*(\beta_{2}^*-\beta_{1})}{q}\|\psi(t)\|_{q}^{q}&=\frac{4\mu\sigma^*(\beta_{2}^*-\beta_{1})}{q}\int_{\mathbb{R}^{N}}|\psi(t)|^{q}dx=\frac{4\mu\sigma^*(\beta_{2}^*-\beta_{1})}{q}\int_{\mathbb{R}^{N}}|\psi|^{\frac{2(4^*-q)}{4^*-2}}|\psi|^{q-\frac{2(4^*-q)}{4^*-2}}dx\\
&\leq\frac{4\mu\sigma^*(\beta_{2}^*-\beta_{1})}{q}\big[\frac{4^*-q}{4^*-2}\int_{\mathbb{R}^{N}}(\eta^{-\frac{q-\frac{2(4^*-q)}{4^*-2}}{4^*}}|\psi|^{\frac{2(4^*-q)}{4^*-2}})^{\frac{4^*-2}{4^*-q}}dx\\
&+\frac{q-\frac{2(4^*-q)}{4^*-2}}{4^*}\int_{\mathbb{R}^{N}}(\eta^{\frac{q-\frac{2(4^*-q)}{4^*-2}}{4^*}}|\psi|^{q-\frac{2(4^*-q)}{4^*-2}})^{\frac{4^*}{q-\frac{2(4^*-q)}{4^*-2}}}dx\big]\\
&\leq\frac{4\mu\sigma^*(\beta_{2}^*-\beta_{1})}{q}\big[C_{\eta}\int_{\mathbb{R}^{N}}|\psi(t)|^{2}dx+\eta|\psi(t)|^{4^*}dx\big].\\	
\end{align*}
Choosing $\eta$ small enough such that
$$\frac{4\mu\sigma^*(\beta_{2}^*-\beta_{1})}{q}\eta<\frac{4\beta^*_{2}(1-\sigma^*)}{4^*},$$
which, together with the definition of $\rho_{0}^*$, implies that
\begin{align*}
Q_{4^*,q}(\psi(t))&\leq4\beta_{2}^*\sigma^* E_{4^*,q}(\psi_{0})+(8-2\beta_{2}^*\sigma^*)\|\Delta\psi(t)\|^{2}_{2}+\frac{4\mu\sigma^*(\beta_{2}^*-\beta_{1})}{q}\|\psi(t)\|_{q}^{q}-\frac{4\beta_{2}^*(1-\sigma^*)}{4^*}\|\psi(t)\|_{4^*}^{4^*}\\	
&\leq4\beta_{2}^*\sigma^* E_{4^*,q}(\psi_{0})-(2\beta_{2}^*\sigma^*-8)\|\Delta\psi(t)\|^{2}_{2}+\frac{4\mu\sigma^*(\beta_{2}^*-\beta_{1})}{q}C_{\eta}\|\psi\|^{2}_{2}\\
&\leq\frac{\beta_{2}^*\sigma^*-4}{2}\rho_{0}-1-\frac{\beta_{2}^*\sigma^*-4}{2}\rho_{0}-\frac{3(\beta_{2}^*\sigma^*-4)}{2}\|\Delta\psi(t)\|^{2}_{2}.\\
\end{align*}
By a similar argument above, we can obtain that
\begin{equation}\label{eq010511}
\frac{d}{dt}M_{\varphi_{R}}[\psi(t)]\leq	-\frac{\beta_{2}^*\sigma^*-4}{2}\|\Delta \psi(t)\|_{2}^{2}.
\end{equation}	
Noticing that
$$|\frac{d}{dt}M_{\varphi_{R}}[\psi(t)]|\geq\min\{\frac{4a^*}{\rho^*},\frac{N(q-2)-8}{4},\frac{4a^*}{\rho_{0}^*},\frac{\beta_{2}^*\sigma^*-4}{2}\}\|\Delta \psi(t)\|_{2}^{2}$$
and
$$M_{\varphi_{R}}[\psi(t)]=M_{\varphi_{R}}[\psi(0)]+\int_{0}^{t}\frac{d}{ds}M_{\varphi_{R}}[\psi(s)]ds,$$
we get that there exists a $t_{0}>0$ such that for any $t>t_{0}$, \eqref{eq604} holds.

\noindent Step $5$:
We now conclude that the solution $\psi(t)$ with the initial data $\psi_{0}=u_{s}$ blows up in finite time. Suppose by contradiction that $T=\infty$. Then integrating \eqref{eq603} on $[t_{0},t]$, we have that
\begin{equation}\label{eq6050}
M_{\varphi_{R}}[\psi(t)]\leq-\int_{t_{0}}^{t}\delta^*\|\Delta \psi(t)\|_{2}^{2}dx.	
\end{equation}
With this estimate at hand, in the same fashion as we did before
the solution $\psi(t,x)$ cannot exist globally. Recalling the blowup alternative we conclude our proof.
\end{proof}

\section{Appendix}\label{Apx}
 We consider the extremal function. Let
\begin{equation}
	U_{\epsilon}=R(\frac{\epsilon}{\epsilon^{2}+|x|^{2}})^{\frac{N-4}{2}},
\end{equation}
where $R=[(N-4)N(N^{2}-4)]^{\frac{N-4}{8}}$.
We know that the best constant $\mathcal{S}$ in the Sobolev embedding $H^{2}(\mathbb{R}^{N})\hookrightarrow L^{4^*}(\mathbb{R}^{N})$, $i.e.$
\begin{equation}\label{eq230213}
\mathcal{S}\|u\|^{2}_{4^*}\leq\|\Delta u\|^{2}_{2}	
\end{equation}
was attained by $U_{\epsilon}$. We also select a suitable constant in order that $U_{\epsilon}$ is the solution of the following critical equation
$$\Delta^{2}u=|u|^{4^*-2}u.$$
Then we have
\begin{equation}\label{eqA.1}
    \begin{cases}
   \|\Delta U_{\epsilon}\|_{2}^{2}=\|U_{\epsilon}\|_{4^*}^{4^*}\\
    \mathcal{S}\|U_{\epsilon}\|_{4^*}^{2}=\|\Delta U_{\epsilon}\|_{2}^{2}\\
    \end{cases}
\end{equation}
it follows that $\|U_{\epsilon}\|_{4^*}^{4^*}=\mathcal{S}^{\frac{N}{4}}$. Then let $\varphi(x)\in C_{0}^{\infty}(\mathbb{R}^{N})$  is a cut-off function satisfying
\begin{description}
	\item $(a)\ 0\leq\varphi(x)\leq1$ and $x\in\mathbb{R}^{N}$.
	\item $(b)\ \varphi(x)\equiv1$ in $B_{1}$, where $B_{s}$ denotes the ball in $\mathbb{R}^{N}$ of center at origin and radius $r$.
	\item $(c)\ \varphi(x)\equiv0$ in $\mathbb{R}^{N}\backslash B_{2}$.
\end{description}
\begin{lemma} \label{lem App.1}
Let $N\geq5$, setting $u_{\epsilon}(x)=U_{\epsilon}(x)\varphi(x)$ and denoting by $\omega$ the area of the unit sphere in $\mathbb{R}^{N}$. We have, for $N\geq5,$
\begin{description}
	\item [$(i)$]
$$\|\Delta u_{\epsilon}\|_{2}^{2}=\mathcal{S}^{\frac{N}{4}}+O(\epsilon^{N-4})\quad\mathrm{and}\quad \|u_{\epsilon}\|_{4^*}^{4^*}=\mathcal{S}^{\frac{N}{4}}+O(\epsilon^{N})$$
	\item [$(ii)$]For some constant $K>0$ and $N\geq5$,
    \[\|u_{\epsilon}\|_{q}^{q}=
    \begin{cases}
    K\epsilon^{N-\frac{(N-4)q}{2}}+o(\epsilon^{N-\frac{(N-4)q}{2}}),\qquad\qquad\qquad\qquad q>\frac{N}{N-4},\\
    K\epsilon^{\frac{N}{2}}|\log\epsilon|+O(\epsilon^{\frac{N}{2}}),\qquad\qquad\qquad\qquad\quad\qquad q=\frac{N}{N-4},\\
    K\epsilon^{\frac{(N-4)q}{2}}+o(\epsilon^{\frac{(N-4)q}{2}}),\qquad\qquad\qquad\qquad\qquad1\leq q<\frac{N}{N-4}.
    \end{cases}\]
 \item [$(iii)$] For some constant $K'>0$
    \[\|u_{\epsilon}\|_{2}^{2}=
    \begin{cases}
    K'\epsilon^{4}+o(\epsilon^{4}),\qquad\qquad\qquad\qquad\qquad\qquad\qquad N>8,\\
    K'\epsilon^{4}|\log\epsilon|+O(\epsilon^{4}),\qquad\qquad\qquad\ \qquad\quad\qquad N=8,\\
    K'\epsilon^{N-4}+o(\epsilon^{(N-4)}),\qquad\qquad\qquad\qquad\qquad5\leq N<8.
    \end{cases}\]
\end{description}
\end{lemma}
\begin{remark}
	Let $0\leq r\leq R$, then
	$$\int_{r\leq|x|\leq R}f(x)dx=\int_{r}^{R}\int_{\sum}f(rx')r^{n-1}dx'dr,$$
	where $\sum=\{x'\in\mathbb{R}^{N},|x'|=1\}$ and $f(x)$ is a continuous function.
\end{remark}
\begin{proof}
$(i)$
\[u_{\epsilon}=\varphi U_{\epsilon}=
\begin{cases}
    U_{\epsilon}, \qquad\qquad x\in B_{1}\\
    0\qquad\qquad\quad x\in\mathbb{R}^{N}\backslash B_{2},\\
    \end{cases}\]
By some calculations, we have
\begin{equation}\label{eq1234}
\Delta u_{\epsilon}=\Delta U_{\epsilon}\varphi+U_{\epsilon}\Delta\varphi+2\nabla U_{\epsilon}\cdot\nabla\varphi.	
\end{equation}
And more accurately, we have
$$\nabla U_{\epsilon}=R\epsilon^{\frac{N-4}{2}}(4-N)\frac{x}{(\epsilon^2+|x|^{2})^{\frac{N-2}{2}}}$$
and
$$\Delta U_{\epsilon}=R\epsilon^{\frac{N-4}{2}}(4-N)\big[\frac{2}{(\epsilon^2+|x|^{2})^{\frac{N-2}{2}}}+(2-N)\frac{|x|^{2}}{(\epsilon^2+|x|^{2})^{\frac{N}{2}}}\big].$$

Passing to radial coordinates and making a change of variable, we get
\begin{align*}
\int_{\mathbb{R}^{N}}|\Delta U_{\epsilon}|^{2}dx&=\epsilon^{N-4}R^{2}(4-N)^{2}\int_{\mathbb{R}^{N}}\frac{4}{(\epsilon^2+|x|^{2})^{N-2}}+(2-N)^{2}\frac{|x|^{4}}{(\epsilon^2+|x|^{2})^{N}}+4(2-N)\frac{|x|^{2}}{(\epsilon^2+|x|^{2})^{N-1}}dx\\	
&=\epsilon^{N-4}R^{2}(4-N)^{2}\omega\int_{0}^{\infty}\frac{4r^{N-1}}{(\epsilon^2+r^{2})^{N-2}}+(2-N)^{2}\frac{r^{N+3}}{(\epsilon^2+r^{2})^{N}}+4(2-N)\frac{r^{N+1}}{(\epsilon^2+r^{2})^{N-1}}dr\\
&=R^{2}(4-N)^{2}\omega\int_{0}^{\infty}\frac{4r^{N-1}}{(1+r^{2})^{N-2}}+(2-N)^{2}\frac{r^{N+3}}{(1+r^{2})^{N}}+4(2-N)\frac{r^{N+1}}{(1+r^{2})^{N-1}}dr.
\end{align*}
While for $\Delta u_{\epsilon}$, recalling \eqref{eq1234}, we have
$$\|\Delta u_{\epsilon}\|^{2}_{2}=\int_{\mathbb{R}^{N}}\big|\Delta U_{\epsilon}\varphi+U_{\epsilon}\Delta\varphi+2\nabla U_{\epsilon}\cdot\nabla\varphi\big|^{2}dx.$$
And it can be decomposed into four parts $I_{1}(\epsilon)$, $I_{2}(\epsilon)$, $I_{3}(\epsilon)$ and $I_{4}(\epsilon)$.
First, we deal with
$$I_{1}(\epsilon)=\int_{\mathbb{R}^{N}}| U_{\epsilon}\Delta\varphi||2\nabla U_{\epsilon}\cdot\nabla\varphi|+|\Delta U_{\epsilon}\varphi||2\nabla U_{\epsilon}\cdot\nabla\varphi|+|\Delta U_{\epsilon}\varphi||U_{\epsilon}\Delta\varphi|dx.$$
Since $\varphi(x)\in C_{0}^{\infty}(\mathbb{R}^{N})$ is cut-off function satisfying $(a),\ (b)$ and $(c)$, there exists constants $K',K''>0$ such that $|\nabla\varphi|\leq K'$ and $|\Delta\varphi|\leq K''$ in $B_{2}\backslash B_{1}$. It follows that the integral in $I_{1}(\epsilon)$ is converging and $I_{1}(\epsilon)=O(\epsilon^{N-4})$.
For the second part, we have
$$I_{2}(\epsilon)=\int_{\mathbb{R}^{N}}|\Delta\varphi|^{2}|U_{\epsilon}|^{2}dx=\epsilon^{N-4}R^{2}\int_{B_{2}\backslash B_{1}}\frac{|\Delta\varphi|^{2}}{(\epsilon^{2}+|x|^{2})^{N-4}}dx.$$
By the definition of $\varphi(x)$, similarly, we can easily get that the integral in $I_{2}(\epsilon)$ is converging and $I_{2}(\epsilon)=O(\epsilon^{N-4})$. While for the third part,
$$I_{3}(\epsilon)=4\int_{\mathbb{R}^{N}}|\nabla U_{\epsilon}|^2|\nabla\varphi\big|^{2}dx=4\epsilon^{N-4}R^{2}(4-N)^{2}\int_{B_{2}\backslash B_{1}}\frac{|x|^2|\nabla\varphi|^2}{(\epsilon^2+|x|^{2})^{N-2}}dx.$$
Similarly, according to the definition of $\varphi(x)$, it follows that the integral in $I_{3}(\epsilon)$ is converging and $I_{3}(\epsilon)=O(\epsilon^{N-4})$.
Finally, we concentrate on
 \begin{align*}
I_{4}(\epsilon)&=\epsilon^{N-4}R^{2}(4-N)^{2}\int_{\mathbb{R}^{N}}\frac{4|\varphi|^{2}}{(\epsilon^2+|x|^{2})^{N-2}}+(2-N)^{2}\frac{|x|^{4}|\varphi|^{2}}{(\epsilon^2+|x|^{2})^{N}}+4(2-N)\frac{|x|^{2}|\varphi|^{2}}{(\epsilon^2+|x|^{2})^{N-1}}dx\\
&=K_{1}(\epsilon)+K_{2}(\epsilon)+K_{3}(\epsilon).	
\end{align*}
Passing to radial coordinates, we get
$$K_{1}(\epsilon)=\epsilon^{N-4}R^{2}(4-N)^{2}\omega\int_{0}^{2}\frac{4|\varphi(r)|^{2}r^{N-1}}{(\epsilon^2+|r|^{2})^{N-2}}dr$$	
that can be decomposed as
\begin{align*}
K_{1}(\epsilon)&=\epsilon^{N-4}R^{2}(4-N)^{2}\omega\int_{0}^{2}\big[\frac{4r^{N-1}(|\varphi(r)|^{2}-1)}{(\epsilon^2+|r|^{2})^{N-2}}+\frac{4r^{N-1}}{(\epsilon^2+|r|^{2})^{N-2}}\big]dr\\
&=L_{1}(\epsilon)+L_{2}(\epsilon).\\
\end{align*}
Since $\varphi(r)\equiv1$ on $[0,1]$, the integral in $L_{1}(\epsilon)$ is converging and $L_{1}(\epsilon)=O(\epsilon^{N-4})$. Then by making a change of variable, we rewrite $L_{2}(\epsilon)$ as
$$L_{2}(\epsilon)=R^{2}(4-N)^{2}\omega\int_{0}^{\frac{2}{\epsilon}}\frac{4r^{N-1}}{(1+r^{2})^{N-2}}dr.$$
Hence we know that
$$K_{1}(\epsilon)=R^{2}(4-N)^{2}\omega\int_{0}^{\frac{2}{\epsilon}}\frac{4r^{N-1}}{(1+r^{2})^{N-2}}dr+O(\epsilon^{N-4}).$$
Similarly, we can also obtain
$$K_{2}(\epsilon)=R^{2}(4-N)^{2}(2-N)^{2}\omega\int_{0}^{\frac{2}{\epsilon}}\frac{r^{N+3}}{(1+r^{2})^{N}}dr+O(\epsilon^{N-4})$$
and
$$K_{3}(\epsilon)=4R^{2}(4-N)^{2}(2-N)\omega\int_{0}^{\frac{2}{\epsilon}}\frac{r^{N+1}}{(1+r^{2})^{N}}dr+O(\epsilon^{N-4})$$
Therefore, as $\epsilon\rightarrow0$, we get that
\begin{align*}
\|\Delta u_{\epsilon}\|_{2}^{2}&=R^{2}(4-N)^{2}\omega	\int_{0}^{\frac{2}{\epsilon}}\big[\frac{4r^{N-1}}{(1+r^{2})^{N-2}}+\frac{(2-N)^{2}r^{N+3}}{(1+r^{2})^{N}}+\frac{4(2-N)r^{N+1}}{(1+r^{2})^{N}}\big]dr+O(\epsilon^{N-4})\\
&=\|\Delta U_{\epsilon}\|_{2}^{2}+O(\epsilon^{N-4}).\\
\end{align*}

Then we deal with the second part of $(i)$. By simple calculations, we have
\begin{align*}
	\|u_{\epsilon}\|_{4^*}^{4^*}&=R^{4^*}\int_{\mathbb{R}^{N}}\varphi^{4^*}(x)(\frac{\epsilon}{\epsilon^{2}+|x|^{2}})^{N}dx\\
&=R^{4^*}\int_{\mathbb{R}^{N}}  \varphi^{4^*}(x)\frac{\epsilon^{N}}{(\epsilon^{2}+|x|^{2})^{N}}dx.
\end{align*}
Passing to radial coordinates, we get
$$\|u_{\epsilon}\|_{4^*}^{4^*}=R^{4^*}\omega\epsilon^{N}\int_{0}^{2}\varphi^{4^*}(r)\frac{r^{N-1}}{(\epsilon^{2}+r^{2})^{N}}dr$$
that can be decomposed as
\begin{align*}
\|u_{\epsilon}\|_{4^*}^{4^*}&=R^{4^*}\omega\epsilon^{N}\int_{0}^{2}\frac{(\varphi^{4^*}(r)-1)r^{N-1}}{(\epsilon^{2}+r^{2})^{N}}dr+R^{4^*}\omega\epsilon^{N}\int_{0}^{2}\frac{r^{N-1}}{(\epsilon^{2}+r^{2})^{N}}dr\\
&=I_{1}(\epsilon)+I_{2}(\epsilon).\\
\end{align*}
Since $\varphi(r)\equiv1$ on $[0,1]$, the integral in $I_{1}(\epsilon)$ is converging and thus $I_{1}(\epsilon)=O(\epsilon^{N})$. Now by making a change of variable, we rewrite $I_{2}(\epsilon)$ as
$$I_{2}(\epsilon)=R^{4^*}\omega\int_{0}^{\frac{2}{\epsilon}}\frac{r^{N-1}}{(1+r^{2})^{N}}dr.$$

Then we notice that
$$\|U_{\epsilon}\|_{4^*}^{4^*}=R^{4^*}\int_{\mathbb{R}^{N}}\frac{\epsilon^{N}}{(\epsilon^{2}+|x|^{2})^{N}}dx.$$
Passing to radial coordinates, we obtain
$$\|U_{\epsilon}\|_{4^*}^{4^*}=R^{4^*}\omega\int_{0}^{\infty}\frac{r^{N-1}}{(1+r^{2})^{N}}dr.$$
As $\epsilon\rightarrow0$, we find that $I_{2}(\epsilon)\rightarrow\|U_{\epsilon}\|_{4^*}^{4^*}=\mathcal{S}^{\frac{N}{4}}$. Hence we know that $\|u_{\epsilon}\|_{4^*}^{4^*}=\mathcal{S}^{\frac{N}{4}}+O(\epsilon^{N})$.

$(ii)$
Here we concentrate on point $(ii)$. We have
$$\|u_{\epsilon}\|_{q}^{q}=R^{q}\int_{\mathbb{R}^{N}}\varphi^{q}(x)(\frac{\epsilon}{\epsilon^{2}+|x|^{2}})^{\frac{(N-4)q}{2}}dx.$$
passing radial coordinates, we get
 $$\|u_{\epsilon}\|_{q}^{q}=R^{q}\omega\epsilon^{\frac{(N-4)q}{2}}\int_{0}^{2}\frac{\varphi^{q}(r)r^{N-1}}{(\epsilon^{2}+r^{2})^{\frac{(N-4)q}{2}}}dr$$
 that can be decomposed as
 \begin{align*}
 \|u_{\epsilon}\|_{q}^{q}&=R^{q}\omega\epsilon^{\frac{(N-4)q}{2}}\int_{0}^{2}\frac{(\varphi^{q}(r)-1)r^{N-1}}{(\epsilon^{2}+r^{2})^{\frac{(N-4)q}{2}}}dr+R^{q}\omega\epsilon^{\frac{(N-4)q}{2}}\int_{0}^{2}\frac{r^{N-1}}{(\epsilon^{2}+r^{2})^{\frac{(N-4)q}{2}}}dr\\
 &=I_{1}(\epsilon)+I_{2}(\epsilon).\\
 \end{align*}
 Since $\varphi(r)\equiv1$ on $[0,1]$, the integral in $I_{1}(\epsilon)$ is converging and thus $I_{1}(\epsilon)=O(\epsilon^{\frac{(N-4)q}{2}})$. Now by making a change of variable, we rewrite $I_{2}(\epsilon)$ as
$$I_{2}(\epsilon)=R^{q}\omega\epsilon^{N-\frac{(N-4)q}{2}}\int_{0}^{\frac{2}{\epsilon}}\frac{r^{N-1}}{(1+r^{2})^{\frac{(N-4)q}{2}}}dr.$$
The integral is converging, as $\epsilon\rightarrow0$, to a finite value if and only if $q>\frac{N}{N-4}$. Hence we have that, for some constant $K>0$,
$$I_{2}(\epsilon)=K\epsilon^{N-\frac{(N-4)q}{2}}+o(\epsilon^{N-\frac{(N-4)q}{2}}).$$
And recording that $I_{1}(\epsilon)=O(\epsilon^{\frac{(N-4)q}{2}})$, we have
$$\|u_{\epsilon}\|_{q}^{q}=K\epsilon^{N-\frac{(N-4)q}{2}}+o(\epsilon^{N-\frac{(N-4)q}{2}}).$$
Now assuming that $q=\frac{N}{N-4}$, and proceeding as before, we get that
 \begin{align*}
 \|u_{\epsilon}\|_{q}^{q}&=R^{q}\omega\epsilon^{\frac{N}{2}}\int_{0}^{2}\frac{(\varphi^{q}(r)-1)r^{N-1}}{(\epsilon^{2}+r^{2})^{\frac{N}{2}}}dr+R^{q}\omega\epsilon^{\frac{N}{2}}\int_{0}^{2}\frac{r^{N-1}}{(\epsilon^{2}+r^{2})^{\frac{N}{2}}}dr\\
 &=I_{1}(\epsilon)+I_{2}(\epsilon).\\
 \end{align*}
with $I_{1}(\epsilon)=O(\epsilon^{\frac{N}{2}})$. Also,
$$I_{2}(\epsilon)=R^{q}\omega\epsilon^{\frac{N}{2}}\int_{0}^{\frac{2}{\epsilon}}\frac{r^{N-1}}{(1+r^{2})^{\frac{N}{2}}}dr=R^{q}\omega\epsilon^{\frac{N}{2}}(|\log\epsilon|+O(1)).$$
Summarizing, for $q=\frac{N}{N-4}$ we obtain that
$$\|u_{\epsilon}\|_{q}^{q}=R^{q}\omega\epsilon^{\frac{N}{2}}|\log\epsilon|+O(\epsilon^{\frac{N}{2}}).$$
It remains to study the case $2\leq q<\frac{N}{N-4}$. We can observe that for all $r>0$,
$$\lim_{\epsilon\rightarrow0}\frac{\varphi^{q}(r)r^{N-1}}{(\epsilon^{2}+r^{2})^{\frac{N-4}{2}q}}=\frac{\varphi^{q}(r)}{r^{(N-4)q-(N-1)}}$$
and also that, for all $\epsilon>0$, for some constant $C$
$$\big|\frac{\varphi^{q}(r)r^{N-1}}{(\epsilon^{2}+r^{2})^{\frac{N-4}{2}q}}\big|\leq\frac{C}{r^{(N-4)q-(N-1)}}\in L^{2}([0,2]).$$
Then from lebesgue's theorem, we can get that
\begin{align*}
\|u_{\epsilon}\|_{q}^{q}&=R^{q}\omega\epsilon^{\frac{(N-4)q}{2}}\int_{0}^{2}\frac{\varphi^{q}(r)r^{N-1}}{(\epsilon^{2}+r^{2})^{\frac{(N-4)q}{2}}}dr\\
&=\epsilon^{\frac{(N-4)q}{2}}R^{q}\omega\int_{0}^{2}\frac{\varphi^{q}(r)}{r^{(N-4)q-(N-1)}}dr+o(\epsilon^{\frac{(N-4)q}{2}}).\\	
\end{align*}
This completes the proof.
\end{proof}

\section*{Acknowledgements}
X. J. Chang was supported by NSFC(11971095).


\end{document}